\documentclass[11pt]{amsart}

\usepackage{amsmath,amssymb,amsfonts,graphics}
\usepackage[all]{xy}
\newcommand{\R}{\mathbb{R}}
\newcommand{\N}{\mathbb{N}}
\newcommand{\Z}{\mathbb{Z}}
\newcommand{\C}{\mathbb{C}}
\newcommand{\hH}{\widehat{H}}
\newcommand{\T}{\mathbb{T}}
\newcommand{\D}{\mathcal{D}}

\newcommand{\G}{\widehat{G}}
\newcommand{\fT}{\mathcal{T}}
\newcommand{\om}{\omega}
\newcommand{\sg}{\sigma}
\newcommand{\Om}{\Omega}
\newcommand{\TOm}{\Theta}

\newcommand{\Del}{\Delta}
\newcommand{\la}{\langle}
\newcommand{\ra}{\rangle}

\newcommand{\supp}{{\rm supp}}
\newcommand{\spn}{{\rm span}}
\newcommand{\cross}{\times}

\newcommand{\bo}{\Box}
\newtheorem{thm}{Theorem}[section]
\newtheorem{cor}[thm]{Corollary}
\newtheorem{lem}[thm]{Lemma}
\newtheorem{prop}[thm]{Proposition}
\theoremstyle{definition}
\newtheorem{defn}[thm]{Definition}
\theoremstyle{remark}
\newtheorem{exm}[thm]{Example}
\newtheorem{rem}[thm]{Remark}
\numberwithin{equation}{section}

\newcommand{\Real}{\mathbb R}
\newcommand{\Realzero}{\mathbb{R}^*}

\newcommand{\norm}[1]{\left\Vert#1\right\Vert}
\newcommand{\abs}[1]{\left\vert#1\right\vert}

\newcommand{\A}{\mathcal{A}}

\newcommand{\B}{\mathcal{B}}

\newcommand{\F}{\mathcal{F}}

\newcommand{\Hi}{\mathcal{H}}
\newcommand{\I}{\mathcal{I}}

\newcommand{\M}{\mathcal{M}}
\newcommand{\n}{\mathcal{N}}

\newcommand{\s}{\mathcal{S}}

\newcommand{\prt}{\widehat{\otimes}}

\begin{document}

\title[Beurling-Fourier Algebras]
{Beurling-Fourier algebras, operator amenability and Arens regularity}

\author{Hun Hee Lee}
\address{Department of Mathematics, Chungbuk National University, 410 Sungbong-Ro, Heungduk-Gu, Cheongju 361-763, Korea}
\email{hhlee@chungbuk.ac.kr}

\author{Ebrahim Samei}
\address{Department of Mathematics and Statistics, University of Saskatchewan, Saskatoon, Saskatchewan, S7N 5E6, Canada}
\email{samei@math.usask.ca}




\begin{abstract}
We introduce the class of Beurling-Fourier algebras on locally compact groups
and show that they are non-commutative analogs of
classical Beurling algebras. We obtain various results with
regard to the operator amenability, operator weak amenability and Arens regularity of
Beurling-Fourier algebras on compact groups and
show that they behave very similarly to the classical Beurling algebras of discrete
groups. We then
apply our results to study explicitly the Beurling-Fourier algebras on $SU(2)$, the 2 $\times$ 2 unitary
group. We demonstrate that how Beurling-Fourier algebras are
closely connected to the amenability of the Fourier algebra of
$SU(2)$. Another major consequence of our results is that
our investigation allows us to construct families of unital infinite-dimensional closed
Arens regular subalgebras of the Fourier algebra of certain products of $SU(2)$.
\end{abstract}

\maketitle

Beurling algebras play an important role in different areas of harmonic
analysis. These are $L^1$-algebras associated to locally compact groups
when we put extra ``weight" on the groups (see Section \ref{S:Beurling alg}). The basic properties of these algebras
are well-known since the works of Beurling \cite{Beu1}, \cite{Beu2}, and Domar
\cite{Dom}, for abelian groups, and Reiter \cite{RS} for the general case
(see also \cite{DL}, \cite{FGLLM}, \cite{G1}, \cite{G2}, \cite{HKK}, and \cite{Sa}). For example,
it is shown in \cite{Dom} that the Beurling algebra $L^1(G,\om)$ is $*$-regular
for $G$ abelian if and only if the weight $\om$ is symmetric and non-quasianalytic. Also various
aspects of cohomologies and Arens regularities of Beurling algebras have been studied by
several authors, most notably Gr{\o}nb{\ae}k \cite{G1}, \cite{G2}, and Dales and Lau \cite{DL}. It is shown
that $L^1(G,\om)$ is amenable as a Banach algebra if and only if $G$
is amenable as a locally compact group and $\{\om(x)\om(x^{-1}) : x\in G\}$
is bounded \cite{G2}. This demonstrates that in most cases, the amenability of Beurling algebras
forces the weight to be trivial. On the other hand, even though the group algebra
$L^1(G)$ is not Arens regular when $G$ is infinite, for a large classes of weights,
it can happen that $\ell^1(G,\om)$ will be Arens regular \cite{DL}.

The aim of the present paper is to develop the corresponding ``dual theory"
for the classical Beurling algebras. That is, we consider the Fourier
algebra $A(G)$ of a locally compact group $G$, and the question of how
could we interpret Beurling algebras in this context and what would be their
basic properties? In the
language of Kac algebras \cite{ES} (or more generally locally
compact quantum groups - see \cite{KV}), $A(G)$ is interpreted as
the dual object of $L^1(G)$ in the sense of generalized Pontryagin
duality.  In particular, when $G$ is abelian, with dual group
$\widehat{G}$, then $A(G) \cong L^1(\widehat{G})$ via the Fourier
transform.  Thus for an abelian group $G$ and a weight $\om$ on $\G$, we
define the Beurling-Fourier algebra $A(G,\om)$ to be the Fourier transform of
Beurling algebra $L^1(\widehat{G}, \om)$ \cite[Section 6.3]{RS}.
In the general non-abelian setting though, $\G$ is not a group and so the extension
of this idea is more delicate!

In order to achieve our goal, we need to focus on the somewhat non-standard
interpretation of the weight $\om$. Consider the co-multiplication
	$$\Gamma : L^\infty(G) \rightarrow L^\infty(G\times G), \; f\mapsto \Gamma f,$$
where
$$\Gamma f(s,t) = f(st).$$
This $\Gamma$ can be easily extended to unbounded Borel measurable functions on $G$
using the same formula.
Now let $\om : G \rightarrow (0,\infty)$ be a continuous function.
Then the submutiplicativity of $\om$ (i.e. $\om$ being a weight) is clearly equivalent to the condition
	\begin{equation}\label{cond-intro}
	\Gamma(\om)(\om^{-1} \otimes \om^{-1}) \le 1.
	\end{equation}

Now let $VN(G)$ be the group von Neumann algebra of $G$,
and let $\Gamma$ be the usual co-multiplication on $VN(G)$ defined by
\begin{equation}\label{cond-intro-2}
\Gamma : VN(G) \rightarrow VN(G\times G), \; \lambda(s) \mapsto \lambda(s) \otimes \lambda(s),
\end{equation}
where $\lambda$ is the left regular representation of $G$.
In Section \ref{S:General cons}, we consider a dual version of weight functions satisfying a dual version of \eqref{cond-intro},
which requires an extension of the $*$-isomorphism $\Gamma$ in (\ref{cond-intro-2}) for certain unbounded operators.
For a fixed representation $VN(G)\subset B(H)$, we define a {\it weight on the dual of $G$} to be a ``suitable" densely defined (possibly unbounded)
operator $W$ acting on $H$ which is affiliated to $VN(G)$ (Definition \ref{Def-weight}).
To simplify our computation, we make a further assumption that $W$ has
a bounded inverse $W^{-1} \in VN(G)$. One major condition that $W$ has to satisfy is the
corresponding dual version of \eqref{cond-intro}:
$$\Gamma(W)(W^{-1}\otimes W^{-1}) \le 1_{VN(G\times G)}.$$
We show that this is the natural extension of a weight on duals of non-abelian
groups. Furthermore, (see Definition \ref{Def-BF-alg}), we define
			$$VN(G,W^{-1}) := \{AW : A \in VN(G)\}$$
		and equip $VN(G,W^{-1})$ with an operator space structure induced by the natural linear
isomorphism
			$$\Phi : VN(G) \rightarrow VN(G,W^{-1}),\; A \mapsto AW.$$
		We will denote the predual of $VN(G,W^{-1})$ by $A(G,W)$ and show that
it is a completely contractive Banach algebra. We call $A(G,W)$ the {\bf Beurling-Fourier
algebras} on $G$.

In the reminder of Section \ref{S:Beurling-Fourier alg}, we show that our approach
allows us to construct various classes of weights on duals of not necessary
abelian groups, namely compact groups and Heisenberg groups.
In Sections \ref{S:cen weight-compact group},
we compute certain {\it central} weights on duals of compact groups. By central weights,
we mean those weights that roughly speaking commute with elements of $VN(G)$
(Definition \ref{Def-weight}).
We show that these central weight on $\G$, the dual of a compact group $G$,
are of the form
\begin{align}\label{Eq:weight-central-compact-intro}
W= \bigoplus_{\pi\in \G}\om(\pi)1_{M_{d_\pi}},
\end{align}
where $\om : \widehat{G} \rightarrow (\delta,\infty)$, for some $\delta >0$, is a function satisfying
\eqref{Eq:weight-central}. In this case, we write $A(G,\om)$ instead of
$A(G,W)$. When $G$ is abelian, the relation \eqref{Eq:weight-central}
is exactly the submultiplicity of $\om$. However we also
construct central weights on duals of non-abelian compact groups using
\eqref{Eq:weight-central}
(see Example \ref{E:weight-compact}). One family of weights which are of particular interest
to us is ($a \geq 0$),
\begin{align}\label{Eq:weight-non-abelian-dim-intro}
\om_a(\pi)=d_\pi^a \ \ \ (\pi \in \G).
\end{align}
We also characterize certain
forms of central weights on duals of Heisenberg groups in terms of
weights on their center (Section \ref{S:cen weight-Heisenberg group} and
Definition \ref{D:weight-Heisenberg-central}).

Section \ref{S:BF alg-compact} is devoted to study operator
amenability, operator weak amenability, and Arens regularity
of the Beurling-Fourier algebra $A(G,\om)$ when $G$ is compact
and $W$ is the central weight (\ref{Eq:weight-central-compact-intro}).
In Section \ref{S:Operator amen-BF alg},
we first compute the operator amenability constant of $A(G,\om)$ for
$G$ finite and use it to characterize the operator amenability of $A(G,\om)$
when $G$ is an arbitrary product of finite groups and $\om$ is the corresponding
weight associated to this product. By applying this result to products of $S_3$,
the permutation group on $\{1,2,3\}$, we construct Beurling-Fourier
algebra with arbitrary operator amenability constant.
This is in contrast to the Fourier algebra of compact group
since the operator amenability constant is always 1 \cite{Ruan}.
We then change our focus and show that for a compact group $G$, $A(G,\om)$ fails to be operator amenable
if $\Om(\pi)=\om(\pi)\om(\overline{\pi}) \to \infty$ whenever $\pi \to \infty$ in the discrete topology.
This provides, for instance, central weights (such as the one defined in (\ref{Eq:weight-non-abelian-dim-intro})) on compact connected semisimple Lie groups whose Beurling-Fourier
algebras are not operator amenable. On some other direction, we show that
$A(G,\om)$ is always operator weakly amenable if $G$ is totally disconnected (Section \ref{S:BF alg-compact-OWA}).
Finally in Section \ref{S:BF alg-compact-Arens regularity}, we present various classes of central weights
whose Beurling-Fourier algebras are Arens regular or fails to be Arens regular.
For instance, we show that $A(G,\om_a)$ is Arens regular if $G$ is a
compact connected semi-simple Lie group and $\om_a$ is a weight satisfying
(\ref{Eq:weight-non-abelian-dim-intro}). All of these results go parallel to the analogous
results in  \cite{DL}, \cite{G1}, and \cite{G2} for classical Beurling algebras.

In Section \ref{S:weight-SU(2)}, we apply the results of the preceding section to study explicitly
Beurling-Fourier algebras on $SU(2)$. We present various classes weights
on $\widehat{SU(2)}$ and show the interesting fact that their Beurling-Fourier
algebras behave vary similarly to the corresponding Beurling algebras
on the $\Z=\widehat{\T}$, where we regard $\T$ as the maximal torus of $SU(2)$.
In Section \ref{S:connection amen A(SU(2))}, we explain in details the intriguing connection between
Beurling-Fourier algebras on $SU(2)$ and the fundamental
work of B. E. Johnson in \cite{J1} on non-amenability of the Fourier algebra $A(G)$ for
a compact connected non-abelian Lie group $G$. We should say that
this was one of the major motivations for us to do this project.

The final Section \ref{S:Arens regular-Fourier alg-SU(2)} is perhaps the most surprising to us
because there are no corresponding results in the classical Beurling algebras!
We construct unital infinite-dimensional closed subalgebras of the Fourier algebra of certain products of $SU(2)$
which are {\it Arens regular}.
We actually show that they are of the form $A(SU(2),\om_{2^n})$, where
$n\in \N$ and $\om_{2^n}$ is the weight defined in (\ref{Eq:weight-non-abelian-dim-intro}).
This is remarkable because this can not happen
for the classical Beurling algebras!
There are unital infinite-dimensional Arens regular Beurling algebras but they can never be closed subalgebras
of some group algebra. These connections are certainly worthwhile further investigations.

In collaboration with M. Ghandehari, we have obtained further results concerning
Beurling-Fourier algebras of Heiesenberg groups $H_d=\C^d \times \R$ ($d\in \N)$
and $n\times n$ special unitary groups $SU(n)$ which will appear
in the subsequent article \cite{GLS}.
We would like to point out that J. Ludwig, N. Spronk, and L. Turowska in \cite{LST} have also considered
and studied the properties of Beurling-Fourier algebras on compact groups.
However they have mainly focued on the question of determining the spectrum of
Beurling-Fourier algebras. Their investigation is parallel to ours and provides a very good complement
to our paper.
 	
\section{Preliminaries}

\subsection{Fourier algebras}\label{S:Fourier alg}
Let $G$\/ be a locally compact group with a fixed left Haar measure.
We denote the group algebra of $G$ with $L^1(G)$.
Given a function $f$\/ on $G$\/ the left and right translation of $f$\/
by $x\in G$\/ is denoted by $(L_xf)(y)=f(xy)$ and
$(R_xf)(y)=f(yx)$, respectively. Let $P(G)$ be the set of all
continuous positive definite functions on $G$\/ and let $B(G)$ be its
linear span. The space $B(G)$ can be identified with the dual of
the group $C^*$-algebra $C^*(G)$, this latter being the
completion of $L^1(G)$ under its largest $C^*$-norm. With
the pointwise multiplication and the dual norm, $B(G)$ is a
commutative regular semisimple Banach algebra. The Fourier
algebra $A(G)$ is the closure of $B(G)\cap C_c(G)$ in $B(G)$. It
was shown in \cite{Em} that $A(G)$ is a commutative regular
semisimple Banach algebra whose carrier space is $G$. Also, if
$\lambda$ is the left regular representation of $G$\/ on $L^2(G)$
then, up to isomorphism, $A(G)$ is the unique predual of $VN(G)$,
the von Neumann algebra generated by the representation $\lambda$.

Let $\G$ be the collection of all equivalence classes of weakly continuous
irreducible unitary representations of $G$ into $B(H_\pi)$ for some Hilbert
space $H_\pi.$ $\G$ can be regarded as the {\it dual} of $G$.
If $G$ is abelain, the $\G$ is the set of continuous characters
from $G$ into $\T=\{ z\in \C \mid |z|=1 \}$ which forms a locally compact abelian group with compact-open
topology. The well-known Fourier transform gives us the identification $L^1(G)\cong A(\G)$
isometrically as Banach algebras.

If $G$ is a compact group, then for all $\pi\in \G$, $H_\pi$ is finite-dimensional.
We denote $d_\pi=\dim H_\pi$, $M_{d_\pi}$ to be the matrix representation of $B(H_\pi)$, and use the convention that $d_\pi$ is the dimension of $\pi$.
If $\pi\in\G$, we fix an orthonormal basis
$\{\xi_1^\pi,\dots,\xi_{d_\pi}^\pi\}$ for $H_\pi$ and define
\begin{equation}\label{eq:piij}
\pi_{ij}:G\to \C \ \ , \ \ \pi_{ij}(s)=\la \pi(s)\xi^\pi_j \mid \xi^\pi_i \ra
\end{equation}
for $i,j=1\dots d_\pi$.  We recall the well-known fact that
\begin{equation}\label{eq:trigfunc}
\fT(G)=\spn\{\pi_{ij}:\pi\in \G ,i,j=1,\dots,d_\pi\}
\end{equation}
is uniformly dense in $C(G)$, the space of continuous functions on $G$.
The Fourier transform on $L^1(G)$ is the one-to-one $*$-linear mapping $\F$ defined by
\begin{equation}\label{Eq:Fourier trans-compact}
\F : L^1(G) \to \bigoplus^\infty_{\pi \in \widehat{G}}M_{d_\pi} \ ,
\ f \mapsto (\widehat{f}(\pi))_{\pi\in \G},
\end{equation}
where $\widehat{f}(\pi)=\displaystyle \int_G f(t)\overline{\pi}(t) dt \in M_{d_\pi}$
and $\overline{\pi}$ is the conjugate representation of $\pi$.
Moreover,
\begin{equation}\label{Eq:Fourier trans-compact-trig poly}
\F(\fT(G))=\Big \{ \bigoplus_{\pi \in F} A_\pi : A_\pi \in M_{d_\pi},
 F\subset \G \ \text{is finite} \Big \}.
\end{equation}
Note that if $A=\bigoplus_{\pi \in F} A_\pi$ with $F\subset \G$ finite,
then
\begin{equation}\label{Eq:Inverse Fourier trans-compact-trig poly}
		f(x) = \F^{-1}(A)(x) = \sum_{\pi \in F} d_\sg \text{tr}(A_\pi \pi(x)),\ \  x\in G.
\end{equation}
Also if we regard $L^1(G)$ as convolution operators
on $L^2(G)$, then $L^1(G)$ is a subalgebra of $VN(G)$ and
$\F$ induces an $*$-isomorphism
\begin{equation}\label{Eq:Fourier trans-compact-VN(G)}
\F : VN(G) \cong \bigoplus^\infty_{\pi \in \widehat{G}}M_{d_\pi}.
\end{equation}
Note that the above direct sums over $\widehat{G}$ assume
the repetition of the same component $d_\pi$-times for $\pi \in \widehat{G}$.
It follows from the preceding identification that
$$A(G)=\{f\in C(G) : \|f\|_{A(G)}=\sum_{\pi \in \G} d_\pi \| \widehat{f}(\pi) \|_1< \infty \},$$
where $\|\cdot\|_1$ is the trace-class
norm on $M_{d_\pi}$.
See \cite[sections 27 and 34]{HR2} for complete
details.

\subsection{Beurling algebras}\label{S:Beurling alg}
Let $G$ be a locally compact group. A weight on $G$ is a continuous function
$\om : G \to (0,\infty )$ such that
$$ \om (st)\leq \om (s)\om (t) \ \ \ (s,t \in G).$$
Sometimes we allow a weight $\om$ just to be measurable and locally finite (i.e. bounded
on every compact subset of $G$), but it is known that (\cite[Theorem 3.7.5]{RS})
for every measurable weight $\om$ there is a continuous weight $\om'$ equivalent to $\om$.

For a (continuous) weight $w$ we define weighted spaces
	$$ L^1(G, \om) := \{ f \;\text{Borel measurable} : \norm{f}_{L^1(G, \om)} = \norm{\om f}_{L^1(G)} < \infty \}$$
and
	$$ L^\infty(G, \frac{1}{\om})
	:= \{ f \;\text{Borel measurable} : \norm{f}_{L^\infty(G,\frac{1}{\om})} = \norm{\frac{f}{w}}_{L^\infty(G)} < \infty \},$$
which are isometric to $L^1(G)$ and $L^\infty(G)$, respectively.
Moreover, $L^\infty(G, \frac{1}{\om})$ is the dual of $L^1(G, \om)$ with the duality bracket
	$$\la f, g \ra = \int_G f(x)g(x)d\mu(x),\;\; f\in L^1(G, \om),\; g\in L^\infty(G, \frac{1}{\om}),$$
where $\mu$ is the left Haar measure on $G$.

For discrete $G$ we denote $L^1(G, \om)$ by $\ell^1(G, \om)$.
With the convolution multiplication $L^1(G, \om)$ becomes a Banach algebra (due to the mutiplicativity of the weight),
and the algebras $L^1(G,\om)$ are called the {\it Beurling algebras} on $G$. For more details see \cite[Chapter 7]{DL}.

\subsection{Operator spaces} We will now briefly remind the reader about the basic properties of operator spaces.
We refer the reader to \cite{ER} for further details concerning the notions presented below.

Let $\mathcal{H}$ be a Hilbert space. Then there is a natural identification between the space
$M_n(B(\mathcal{H}))$ of $n\times n$ matrices with entries in $B(\mathcal{H})$ and the space
$B(\mathcal{H}^n)$. This allows us to define a sequence of norms $\{\norm{\cdot}_n \}$ on the spaces
$\{M_n(B(\mathcal{H}))\}$. If $V$ is any subspace of $B(\mathcal{H})$, then the spaces
$M_n(V)$ also inherit the above norm. A subspace $V\subseteq B(\mathcal{H})$ together with the family
$\{\norm{\cdot}_n \}$ of norms on $\{M_n(V)\}$ is called a \textit{concrete operator space}. This
leads us to the following abstract definition of an operator space:

\begin{defn}
An operator space is a vector space $V$ together with a family $\{\norm{\cdot}_n\}$ of Banach space norms on $M_n(V)$ such that for each
$A\in M_n(V),B\in M_m(V)$ and $[a_{ij}],[b_{ij}]\in M_n(\mathbb{C})$

\[
\begin{array}{ll}
i) & \norm{ \left[
\begin{array}{ll}
A & 0 \\
0 & B
\end{array}
\right] }_{n+m}=\max \{\norm{A}_n, \norm{B}_m\} \\
& \\
ii) & \norm{ [a_{ij}]A[b_{ij}] }_n \leq \norm{[a_{ij}]} \norm{A}_n \norm{[b_{ij}]}
\end{array}
\]

Let $V,W$ be operator space, $\varphi :V\rightarrow W$ be linear. Then
\[
\norm{ \varphi }_{cb}=\sup_n \{\norm{ \varphi _{n}} \}
\]
where $\varphi _{n}:M_{n}(V)\rightarrow M_{n}(W)$ is given by
\[
\varphi _{n}([v_{ij}])=[\varphi (v_{ij})].
\]

We say that $\varphi $ is completely bounded if $\norm{ \varphi }_{cb}<\infty ;$
is completely contractive if $\norm{ \varphi }_{cb}\leq 1$ and is a complete isometry if each $\varphi _{n}$ is an isometry.

Given two operator spaces $V$ and $W$, we let $CB(V,W)$ denote the space of all completely
bounded maps from $V$ to $W$. Then $CB(V,W)$ becomes a Banach space with respect to the norm
$\norm{ \cdot }_{cb}$ and is in fact an operator space via the identification $M_n(CB(V,W))\cong
CB(V,M_n(W))$.
\end{defn}

It is well-known that every Banach space can be given an operator
space structure, though not necessarily in a unique way. It is also
clear that any subspace of an operator space is also an operator space with respect to the inherited norms.
Moreover, for duals and preduals of operator spaces, there are canonical operator space structures.
As such the predual of a von Neumann algebra
and the dual of a $C^*$-algebras respectively, the Fourier and Fourier-Stieltjes algebras
inherit natural operator space structures.

Given two Banach spaces $V$ and $W$, there are many ways to define a norm on the
algebraic tensor product $V\otimes W$. Distinguished amongst such norms is the \textit{Banach space
projective tensor product norm} which we denote by $V\otimes ^\gamma W$. A fundamental property of
the projective tensor product is that there is a natural isometry between $(V\otimes ^\gamma W)^*$ and
$B(V,W^*)$. Given two operator spaces $V$ and $W$, there is an operator space analog of the projective
tensor product norm which we denote by $V\widehat{\otimes}\,W$. In this case, we have a natural
complete isometry between $(V\widehat{\otimes}\,W)^*$ and $CB(V,W^*)$.

\begin{defn}

A Banach algebra $A$ that is also an operator space is called a
\textit{completely contractive Banach algebra} if the multiplication map
\[m:A\widehat{\otimes}A\rightarrow A,\; u\otimes v \mapsto uv\]
is completely contractive. In particular, both $B(G)$ and $A(G)$ are
completely contractive Banach algebras (see \cite{Em}).

Let $A$ be a completely contractive Banach algebra. An operator space $X$ is called a
\textit{completely bounded $A$-bimodule}, if $X$ is a Banach $A$-bimodule and if the maps
\[A\widehat{\otimes} X\rightarrow X \ \ , \ \ u\otimes x \mapsto ux \]
and
\[X\widehat{\otimes} A\rightarrow X \ \ , \ \ x\otimes u \mapsto xu \]
are completely bounded. In general, if $X$ is a completely bounded $A$-bimodule, then its dual
space $X^*$ is a completely bounded $A$-bimodule via the actions
\[(u\cdot T)(x)=T(xu) \ \ \ , \ \ \ (T\cdot u)(x)=T(ux)\]
for every $u\in A$, $x\in X$, and $T\in X^*$.

$A$ is \emph{operator amenable} if, for every completely contractive
Banach $X$-bimodules, every completely bounded derivation from $A$ into $X^*$ is inner.
One characterization of operator amenability is that $A$ is operator amenable if and only if
it has a \emph{vertual diagonal} \cite{J3} i.e. there is $M\in (A \widehat{\otimes} A)^{**}$ such that
$$a\cdot M=M \cdot a \ \ , \ \ \ a m^{**}(M)= m^{**}(M)a=a \ \ (a\in A),$$
where $a\cdot (b\otimes c)=ab \otimes c$, $(b\otimes c)\cdot a=b \otimes ca$, and
$\pi : A\otimes A \to A$ is the multiplication operator.
$A$ is \emph{operator weakly amenable} if every
completely bounded derivation from $A$ into $A^*$ is inner \cite{FW}.
\end{defn}

\subsection{Arens regular Banach algebras}

Let $A$ be a (completely contractive) Banach algebra. We can define
two products on $A^{**}$, the second dual of $A$, known as the {\it first and second
Arens products} as follows: For every $F,E\in A^{**}$ with
$F = w^*-\lim_\alpha f_\alpha$, and $E=w^*-\lim_\beta g_\beta$, $\{f_\alpha\}$, $\{g_\beta\}\subset A$,
we let the first (second) Arens product be
	$$F\bo E= w^*-\lim_\alpha \lim_\beta f_\alpha g_\beta \;\;\text{and}\;\; F \diamond E= w^*-\lim_\beta \lim_\alpha  g_\beta f_\alpha.$$
We say that $A$ is {\bf Arens regular} if the first and second Arens products
always coincide i.e.
	$$F\bo E=F\diamond E, \;\;\forall F,E\in A^{**}.$$
	
If $A$ is Arens regular, then every closed subalgebra of $A$ or a quotient of $A$ is also Arens regular.
It is well-known that C$^*$-algebras (or more generally, operator algebras) are Arens regular. However the group
algebra $L^1(G)$ is Arens regular if and only if $G$ is finite \cite{DL}.

Also the Arens regularity of the Fourier algebra $A(G)$ implies that $G$ is discrete, non-amenable, and does not
contain a copy of $\mathbb{F}_2$, the free group on two generators \cite{B1}, \cite{B2}.
It is still an open question whether the Arens regularity of $A(G)$ implies that $G$ is finite.

\section{Beurling-Fourier algebra on a locally compact group}\label{S:Beurling-Fourier alg}

\subsection{General construction}\label{S:General cons}

We begin the construction of a dual object of classical Beurling algebras
by the following reformulation of the multiplicativity of weight functions.

Let $G$ be a locally compact group and recall the co-multiplication
	$$\Gamma : L^\infty(G) \rightarrow L^\infty(G\times G), \; f\mapsto \Gamma f,$$
where $\Gamma f(s,t) = f(st)$.
This $\Gamma$ can be easily extended for unbounded Borel measurable functions on $G$
using the same formula.

Now let $\om : G \rightarrow (0,\infty)$ be a continuous function.
Then the submultiplicativity of $\om$ is clearly equivalent to the condition
	\begin{equation}\label{cond}
	\Gamma(\om)(\om^{-1} \otimes \om^{-1}) \le 1.
	\end{equation}

Our aim is first to define a dual version of weight functions satisfying a dual version of \eqref{cond},
which requires an extension of a $*$-isomorphism for certain unbounded operators.
We will describe the process in the following lemma. We refer the reader to \cite[Chapters X.1 and
X.2]{Con} and \cite[Chapter 5.5.6]{KR83} for the definition and basic properties of unbounded
operators.

	\begin{lem}\label{extension}
	Let $\M \subseteq B(H)$ and $\n \subseteq B(K)$ be von Neumann algebras and $\Phi:\M\rightarrow \n$ be a $*$-isomorphism.
	We suppose that
		\begin{enumerate}
			\item there is an increasing net of projections $(E_i)_{i\in \I} \subseteq \M$ such that
			$\D := \bigcup_{i\in \I}E_i(H)$ is dense in $H$,
			\item there is a closed operator $W$ on $H$ with the domain containing $\D$ such that $WE_i$'s are bounded self-adjoint operators in $\M$, and
			\item $\D' := \bigcup_{i\in \I}\Phi(E_i)(K)$ is dense in $K$.
		\end{enumerate}
	Then, the linear operator $B$ defined on $\D'$ by
		$$B(k) := \Phi(WE_i)(k)\;\;\text{for}\;\; k\in \Phi(E_i)(K)$$
	is a closable operator on $K$, whose closure is self-adjoint.
	\end{lem}
\begin{proof}
Since $(E_i)_{i\in \I}$ is increasing, $(\Phi(E_i))_{i\in \I}$ is also an increasing net of projections in $\n$, so that $B$ is well-defined.
Now we can apply the same argument as in \cite[Lemma 5.6.1]{KR83} to show that
$B$ is closable with the self-adjoint closure acting on $K$.
\end{proof}
	
	\begin{defn}\label{Def-extension}
	Suppose that we are in the same situation as in Lemma \ref{extension}.
	We define $\Phi(W)$ acting on $K$ by $\Phi(W) := \overline{B}$,
	where $\overline{B}$ is the closure of $B$.
	\end{defn}

	\begin{rem}
	(1) The above definition of $\Phi(W)$ is an extension of $\Phi$ in the following sense.
	If $W$ is bounded with the domain $H$, then $\Phi(W)$ defined in Definition \ref{Def-extension} (denoted by $T$)
	and the original $\Phi(W)$ (denoted by $S$) coincide on a dense subspace of $K$.
	Indeed, if we put $W_i = WE_i$, $i\in \I$, where $E_i$'s are the projection in Lemma \ref{extension}, then
	we have $W_i \rightarrow W$ strongly, and so, $W\in \M$. Moreover since $W_i$'s and $W$ are uniformly bounded, we have actually $W_i \rightarrow W$ $\sigma$-strongly.
	Thus, $\Phi(W_i) \rightarrow S$ $\sigma$-strongly. From the definition it is clear that $\Phi(W_i)x \rightarrow Tx$ for all $x\in \D'$,
	so that $Tx=Sx$ for all $x\in \D'$, and $\D'$ is dense in $K$.
	
	(2) We will use the convention that if two bounded operators $S$, $T$, acting on a Hilbert space,
	coincide on a dense subspace, then we identify $S$ and $T$, and we use the notation $S=T$.
	\end{rem}

Now we go back to the definition of a dual version of weight functions.
Let $VN(G)$ be the group von Neumann algebra,
and let $\Gamma$ be the usual co-multiplication on $VN(G)$ defined by
	$$\Gamma : VN(G) \rightarrow VN(G\times G), \; \lambda(s) \mapsto \lambda(s) \otimes \lambda(s),$$
where $\lambda$ is the left regular representation of $G$.
Recall that a densely defined (possibly unbounded) operator $T$ acting on $H$ is said to be {\it affiliated to $\M$}, a von Neumann algebra in $B(H)$,
if $UTU^* = T$ for any unitary $U \in \M'$ \cite[Chapter 5.5.6]{KR83},
and that $T$ is called {\it boundedly invertible} if there is a bounded operator $S:H\rightarrow H$
such that $TS = id_H$ and $ST \subseteq id_H$ \cite[1.14 Definition]{Con}. In the latter case, the choice
of $S$ is unique so we denote $S$ by $T^{-1}$
and call it the {\it bounded inverse of $T$}.
	
	\begin{defn}\label{Def-weight}
	Let $G$ be a locally compact group, and let $VN(G) \subseteq B(H)$ be a fixed representation of $VN(G)$.
	A closed densely defined positive operator $W$ on $H$ affiliated to $VN(G)$ with the bounded inverse $W^{-1} \in VN(G)$
	is called a {\bf weight on the dual of $G$} if
		\begin{enumerate}
			\item $W$ satisfies the conditions in Lemma \ref{extension} with
			$\M=VN(G)$, $\n=VN(G\times G)\subseteq B(H\otimes_2 H)$, and $\Phi=\Gamma$,
			\item $\D_0 := \{ x\in H\otimes_2 H: (W^{-1}\otimes W^{-1})x \in \D'\}$ is dense in $H\otimes_2 H$,
			\item $\Gamma(W)(W^{-1}\otimes W^{-1})$ is bounded on $\D_0$
			(we still denote its unique extension to $H\otimes_2 H$ by $\Gamma(W)(W^{-1}\otimes W^{-1})$),
			\item $\Gamma(W)(W^{-1}\otimes W^{-1}) \le 1_{VN(G\times G)},$
and
			\item $VN(G)W^{-1}:=\{AW^{-1} : A\in VN(G) \}$ is $w^*$-dense in $VN(G)$.
		\end{enumerate}
	We say that a weight $W$ on the dual of $G$ is {\bf central} if $WE_i \in VN(G)'$ for any $i\in \I$,
	where $(E_i)_{i\in \I}$ is the net of projections in Lemma \ref{extension}.
	\end{defn}

	\begin{rem}\label{R:weight-basis properties}

(1) In this paper, we will usually exploit the representation of $VN(G)$
	coming from the representation theory of the group $G$ in the concrete examples,
	namely the case of compact groups and the case of Heisenberg groups.
	
(2) We require our weight $W$ to be boundedly invertible
in order to avoid unnecessary difficulties of unbounded inverses.
Of course, we sacrifice some generality here,
but all of our examples show that this is a reasonable restriction.

	\end{rem}

		\begin{defn}\label{Def-BF-alg}
		For a weight $W$ on the dual of $G$ we define
\begin{align}\label{Eq:weight-VN(G,W-1)}
			VN(G,W^{-1}) := \{AW : A \in VN(G)\}.
\end{align}
Hence each element of $VN(G,W^{-1})$ is a densely defined operator on $H$.\
We put the canonical linear structure on $VN(G,W^{-1})$.
Since $W^{-1} \in VN(G)$, it follows that the mapping
\begin{align}\label{Eq:isometry-VN(G) and weight}
\Phi : VN(G) \rightarrow VN(G,W^{-1}),\; A \mapsto AW
\end{align}
is a linear isomorphism. We endow an operator space structure on $VN(G,W^{-1})$
so that $\Phi$ induces a complete isometry. In particular,
		$$\norm{AW}_{VN(G,W^{-1})} = \norm{A}_{VN(G)}.$$
		We will denote the predual of $VN(G,W^{-1})$ by $A(G,W)$.
		
		Finally we define $C^*_r(G,W^{-1})$ by
		$$C^*_r(G,W^{-1}) := \{AW : A \in C^*_r(G)\}.$$
		Clearly $\Phi|_{C^*_r(G)}$ is a complete isometry between $C^*_r(G)$ and $C^*_r(G,W^{-1})$.

		\end{defn}

	\begin{rem}\label{R:weighted spaces}
			(1) The above definition of $A(G,W)$ is an abstract one,
			but we have a natural realization of $A(G,W)$ as follows.
			For any $\phi\in A(G)$, $W^{-1}\phi$ is an element in $A(G)$ satisfying
				$$(W^{-1}\phi)(A) = \phi(AW^{-1}), \ \ A\in VN(G).$$
			Hence we have
\begin{align}\label{Eq:weight-A(G,W)}
				A(G,W) = \{ W^{-1}\phi : \phi\in A(G) \}
\end{align}
			with the duality bracket
				\begin{align}\label{Eq:weight-duality bracket}
				\la W^{-1}\phi, AW \ra = \phi(A)
				\end{align}
			for $\phi \in A(G)$ and $A\in VN(G)$.
			Moreover, $\Phi$ is $w^*$-$w^*$ continuous and its preadjoint
			$\Phi_* : A(G,W) \rightarrow A(G)$ is given by
			$$\Phi_*(W^{-1}\phi) = \phi.$$	

			(2) The condition (5) of Definition \ref{Def-weight} is redundant if the
weight $W$ is central. Indeed, $VN(G)W^{-1}$ is $w^*$-dense in $VN(G)$ if
and only if the map $A(G) \to A(G)$, $\varphi \mapsto W^{-1}\varphi$
is one-to-one. Now suppose that $W^{-1}\varphi=0$. Then
$WE_iW^{-1}\varphi=0$, $i\in \I$, where $E_i$'s are the projection in Lemma \ref{extension}.
However $WE_i\in VN(G)'$, and so, $WE_iW^{-1}\varphi=W^{-1}\varphi WE_i=E_i \to 1_{VN(G)}$
strongly. Hence $\varphi=0$.
			
			(3) Since $W^{-1} \in VN(G)$, the inclusion map (or the formal identity)
			$j : VN(G) \to VN(G,W^{-1}),\; A \mapsto (AW^{-1})W$ is
			a completely bounded $w^*$-$w^*$ continuous map with $\norm{j}_{cb} \le \norm{W^{-1}}.$
			Moreover, $j$ has a dense range since $\Phi^{-1}\circ j : VN(G)\rightarrow VN(G)$,
 $A \mapsto AW^{-1}$ has a dense range by Definition \ref{Def-weight} (5). This implies that
 the preadjoint of $j$, $j_* : A(G,W) \to A(G)$ is completely bounded and one-to-one.
			Note that $j_*$ is clearly the formal identity.
			Thus we can (and will) assume that $A(G,W) \subseteq A(G)$
			and view any element $\phi \in A(G,W)$ as a continuous function on $G$ vanishing at infinity.
			
			(4) We do not know whether $W\otimes W$ always defines a weight on the dual
of $G\times G$. Nevertheless we can formally define
$VN(G\times G, W^{-1}\otimes W^{-1})$ and $A(G\times G, W\otimes W)$ similar to
(\ref{Eq:weight-VN(G,W-1)}) and (\ref{Eq:weight-A(G,W)}), respectively.
This induces the natural complete isometry
			\begin{align}\label{Eq:isometry-VN(G) and weight-tensor}
			\Psi : VN(G\times G) \rightarrow VN(G\times G,W^{-1}\otimes W^{-1}),\; A\otimes B \mapsto (A\otimes B)(W\otimes W).
			\end{align}
In fact, we can identify
				$$\big(A(G,W)\prt A(G,W)\big)^* = VN(G\times G, W^{-1}\otimes W^{-1}).$$
			Indeed, from (\ref{Eq:isometry-VN(G) and weight}) and (\ref{Eq:isometry-VN(G) and weight-tensor}) we have the following composition of complete isometries
				$$A(G,W)\prt A(G,W) \stackrel{\Phi_*\otimes \Phi_*}{\longrightarrow} A(G)\prt A(G) \cong A(G\times G) \stackrel{\Psi^{-1}_*}{\longrightarrow} A(G\times G, W\otimes W),$$
			which can be easily checked to be the formal identity.
			
	\end{rem}

Now we would like to endow a completely contractive Banach algebra structure on $A(G,W)$.
Recall that the Banach algebra structure of $A(G)$ comes from the co-multiplication $\Gamma$, so that
we will consider an appropriate map $VN(G,W^{-1})\rightarrow VN(G\times G, W^{-1}\otimes W^{-1})$,
which is essentially the extension of $\Gamma$.
By (3) in Definition \ref{Def-BF-alg} we have a normal complete contraction
	$$\widetilde{\Gamma} : VN(G) \rightarrow VN(G\times G),$$
	defined by
	$$A \mapsto \Gamma(A)\Gamma(W)(W^{-1}\otimes W^{-1}).$$
We define the $w^*$-$w^*$ continuous complete contraction $$\Gamma^W : VN(G,W^{-1})\rightarrow VN(G\times G, W^{-1}\otimes W^{-1})$$ by
	\begin{equation}\label{modified-co-multiplication}
	\Gamma^W := \Psi \circ \widetilde{\Gamma} \circ \Phi^{-1}.
	\end{equation}
We can say that $\Gamma^W$ is essentially an extension of $\Gamma$ in the following sense.

\begin{thm}\label{T:extension of co-product}
Let $G$ be a locally compact group, and let $W$ be a weight on the dual of
$G$. Then the following diagram is commutative:
\[
\xymatrix{
VN(G) \ar@<.5ex>[rr]^{\Gamma} \ar@<.5ex>[d]^{j}
& & VN(G\times G) \ar[d]^{j \otimes j}   \\
VN(G,W^{-1}) \ar@<.5ex>[rr]^{\Gamma^W} & & VN(G\times G,W^{-1}\otimes W^{-1})
. }
\]
\end{thm}

\begin{proof}
It suffices to show that for every $A\in VN(G)$,
\begin{equation}\label{Eq:extension of co-product}
\Gamma^W(A)x = \Gamma(A)x
\end{equation}
for all $x\in \D(W)\otimes \D(W)$, where $\D(W)$ is the domain of $W$.
Let $W_i = WE_i$, $i\in \I$, where $E_i$'s are the projection in Lemma \ref{extension}. Then we have
$W^{-1}W_i x \rightarrow x$ for all $x\in \D=\bigcup_{i\in I} E_i(H)$. Since $W^{-1}W_i$'s are uniformly bounded and $\D$ is dense in $H$,
we have $W^{-1}W_i \rightarrow 1_{VN(G)}$ $\sigma$-strongly.
Thus $\Gamma(W^{-1})\Gamma(W_i) \rightarrow 1_{VN(G\times G)}$ $\sigma$-strongly.
Since $$\Gamma(W_i)x \rightarrow \Gamma(W)x$$ for all $x\in \D'=\bigcup_{i\in I} \Gamma(E_i)(H\otimes_2 H)$, we have $\Gamma(W^{-1})\Gamma(W)x=x$ for all $x\in \D'$,
so that for every $A\in VN(G)$
	\begin{align*}
	\Gamma(AW^{-1})\Gamma(W)(W^{-1}\otimes W^{-1})x & = \Gamma(A)\Gamma(W^{-1})\Gamma(W)(W^{-1}\otimes W^{-1})x\\
	& = \Gamma(A)(W^{-1}\otimes W^{-1})x
	\end{align*}
for all $x\in \D_0= \{ x\in H\otimes_2 H: (W^{-1}\otimes W^{-1})x \in \D'\}$. Since $\D_0$ is dense in $H\otimes_2 H$ ((2) of Definition \ref{Def-weight}) and both operators are bounded  ((3) of Definition \ref{Def-weight}), we have
$$\Gamma(AW^{-1})\Gamma(W)(W^{-1}\otimes W^{-1}) = \Gamma(A)(W^{-1}\otimes W^{-1}).$$
Thus
\begin{align*}
	\Gamma^W(A) &= \Psi(\widetilde{\Gamma}(\Phi^{-1}(A))) \\
	&= \Psi(\widetilde{\Gamma}(AW^{-1})) \\
	&= \Psi(\Gamma(AW^{-1})\Gamma(W)(W^{-1}\otimes W^{-1})) \\
	 & = \Psi(\Gamma(A)(W^{-1}\otimes W^{-1}))\\
	& = \Gamma(A)(W^{-1}\otimes W^{-1})(W\otimes W).
	\end{align*}
	Hence (\ref{Eq:extension of co-product}) follows.
	\end{proof}

We are now ready to define a suitable completely contractive Banach algebra structure on $A(G,W)$.
Indeed, since $\Gamma^W$ is a complete contraction and also a $w^*$-$w^*$ continuous mapping,
the preadjoint $\Gamma^W_*$ of $\Gamma^W$ defines a completely contractive Banach algebra structure on $A(G,W)$. This will allow us to present the following definition.

\begin{defn}\label{D:weight-Beurling Fourier alg}
Let $G$ be a locally compact group, and let $W$ be a weight on the dual of $G$.
The completely contractive Banach algebra $A(G,W)$ defined in Definition \ref{Def-BF-alg}
with the multiplication
$$\Gamma^W_* : A(G,W) \widehat{\otimes} A(G,W) \to A(G,W)$$
is called the {\bf Beurling-Fourier algebra} on $G$.

We will use the notation
	$$\phi\, \cdot_{A(G,W)} \psi = \Gamma^W_*(\phi \otimes \psi),\;\; \phi,\psi\in A(G,W),$$
while
	$$\phi \cdot_{A(G)} \psi = \Gamma_*(\phi \otimes \psi),\;\; \phi,\psi\in A(G).$$
\end{defn}
	
	\begin{rem}

		(1) It follows from the commuting diagram in Theorem \ref{T:extension of co-product} that
the following diagram is also commutative:
\[
\xymatrix{
A(G\times G, W\otimes W) \ar@<.5ex>[rr]^{\Gamma_*^W} \ar@<.5ex>[d]^{\iota_*\otimes \iota_*}
& & A(G,W) \ar[d]^{\iota_* }   \\
A(G\times G) \ar@<.5ex>[rr]^{\Gamma_*} & & A(G)
. }
\]
This implies that for every $\phi, \psi \in A(G,W)$,
$$\phi\, \cdot_{A(G,W)} \psi = \Gamma^W_*(\phi \otimes \psi)
= \Gamma_*(\phi \otimes \psi)=\phi \cdot_{A(G)} \psi,$$
or equivalently, the multiplication on $A(G,W)$
can be be understood as the pointwide multiplication of continuous functions
so that $A(G,W)$ can be viewed as a subalgebra of $A(G)$.

		(2) The definition of the Banach algebra structure on $A(G,W)$ for a weight $W$ on the dual of $G$ is somewhat technical
		since we are working with general unbounded operators.
		If $W$ is bounded or at least $VN(G)$ is semifinite with a trace $\tau$ and $W$ is $\tau$-measurable,
		then the above construction becomes much easier, since the extension of $*$-isomorphism can be easily understood (\cite[Lemma 2.4]{PS}).
		However, the weight $W$ we are interested in is usually pretty much unbounded,
		so that $W$ is not even $\tau$-measurable.

	\end{rem}

\subsection{Central weights on the dual of compact groups}\label{S:cen weight-compact group}

	We will show in this section how we can construct central weights on
	the duals of compact groups. We will see, eventually,
	that they are a generalization of classical weights on discrete groups.
	
	Let $G$ be a compact group. Then from (\ref{Eq:Fourier trans-compact-VN(G)}),
		$$VN(G) \cong \bigoplus_{\pi \in \widehat{G}}M_{d_\pi} \subseteq B(H),$$
	where $H = \bigoplus_{\pi \in \widehat{G}} \ell^2_{d_\pi}$.
	Note that the above direct sums over $\widehat{G}$ assume the repetition of the same component $d_\pi$-times for $\pi \in \widehat{G}$. For the rest of this
article, we always consider the above representation of $VN(G)$.
	
	Before proceeding further we need to know how the co-multiplication on $VN(G)$ is translated in the above representation of $VN(G)$.
	Note that the left regular representation $\lambda$ has the decomposition $\lambda \cong \bigoplus_{\pi \in \widehat{G}} \overline{\pi}.$
	Consider a central element $W\in VN(G)$ defined by
		$$W = \bigoplus_{\pi \in \widehat{G}}\om(\pi)1_{M_{d_\pi}},$$
	where $\om(\pi)$'s are positive numbers and $F = \{\pi \in \widehat{G} : \om(\pi)>0\}$ is a finite set. Then from (\ref{Eq:Fourier trans-compact}), (\ref{Eq:Fourier trans-compact-trig poly}) and the Fourier inversion formula
 (\ref{Eq:Inverse Fourier trans-compact-trig poly}) we have that
 $$\F(f) = (\widehat{f}(\pi))_{\pi\in \G}= W \ \text{or}\ W = \F\Big( \int_G f(x) \lambda(x) dx\Big),$$
 where
 $$f(x) = \sum_{\sigma \in \widehat{G}} d_\sigma \om(\sigma)\text{tr}(\sigma(x)),\; (x\in G).$$		
	Thus
		\begin{align*}
		\Gamma(W) & = \F\Big( \int_G f(x) \lambda(x) \otimes \lambda(x) dx \Big)\\
		& = \bigoplus_{\pi, \pi' \in \widehat{G}} \int_G f(x) \overline{\pi(x)} \otimes \overline{\pi'(x)}dx\\
		& = \bigoplus_{\pi, \pi' \in \widehat{G}}\bigoplus^{N}_{k=1} \int_G f(x) \overline{\tau^k(x)} dx,
		\end{align*}
	where
	\begin{equation}\label{Eq:tensor formula}
	\pi\otimes \pi' \cong \bigoplus^N_{k=1}\tau^k
	\end{equation}
	for some $(\tau^k)^N_{k=1} \subseteq \widehat{G}$.
	Note that we are allowing the repetition of $\tau^k$'s,
	so that it is possible that $\tau^k\cong \tau^l$ for some $k\neq l$.
	By the Schur orthogonality relation,
		\begin{align*}
		\Gamma(W)
		& = \bigoplus_{\pi, \pi' \in \widehat{G}}\bigoplus^{N}_{k=1} \int_G \sum_{\sigma} d_\sigma w(\sigma)
		\sum^{d_\sigma}_{i=1}\sigma_{ii}(x)\overline{\tau^k(x)}dx\\
		& = \bigoplus_{\pi, \pi' \in \widehat{G}}\bigoplus^{N}_{k=1} w(\tau^k)1_{M_{d_{\tau^k}}}.
		\end{align*}
		
We can change the order of the direct sum using the following notation.
	\begin{defn}\label{def-supp-rep}
	Let $\rho$ be a continuous finite-dimensional (unitary) representation of $G$.
	We recall that the {\bf support} of $\rho$ in $\G$ is the (finite) set of continuous
	finite-dimensional irreducible unitary representation of $G$ that appear
	in the decomposition of $\rho$, i.e.
	$$\supp\, \rho=\{ \tau_i\in \G \mid \rho \cong \oplus_{i=1}^n \tau_i \}.$$
	\end{defn}
Using the preceding definition and the fact that $F = \{\pi \in \widehat{G} : \om(\pi)>0\}$, we can write
	\begin{equation}\label{Eq:co-multi}
	\Gamma(W) = \bigoplus_{\sigma \in F} \bigoplus_{\substack{\pi,\pi'\in \widehat{G} \\ \sigma \in\, \supp\, \pi\otimes \pi'}}
	w(\sigma)1_{M_{d_\sigma}}.
	\end{equation}
	
Now we consider a function $\om : \widehat{G} \rightarrow (\delta,\infty)$ for some $\delta>0$.
We would like to construct a central weight associated to $\om$.
Let $\F$ be the set of all finite subset of $\widehat{G}$ directed by the inclusion.
For every $F\in \F$, let $E_F$ be the projection in $VN(G)$ defined by
	$$E_F = \bigoplus_{\pi \in F}1_{M_{d_\pi}}.$$
It is clear that $(E_F)_{F\in \F}$ is an increasing net of projections in $VN(G)$ and $\D = \bigcup_{F\in \F}E_F(H)$ is dense in $H$.
Let $W_F$ be the operator in $VN(G)$ given by
	$$W_F := \bigoplus_{\pi \in F}\om(\pi)1_{M_{d_\pi}}.$$
Consider the linear operator $W_0$ with the domain $\D$ defined by
	$$W_0(h) := W_F(h),\;\; h\in E_F(H).$$
If we apply the same argument as in \cite[Lemma 5.6.1]{KR83}, then we can show that
$W_0$ is closable with the self-adjoint closure. We will denote this closure by
	\begin{equation}\label{W-def}
	W = \bigoplus_{\pi \in \widehat{G}}\om(\pi)1_{M_{d_\pi}}.
	\end{equation}
We can exactly determine when $W$ is a weight on the dual of $G$.

\begin{thm}
Let $G$ be a compact group, and let $\om : \widehat{G} \rightarrow (\delta,\infty)$
be a function, where $\delta>0$. The operator $W$ constructed in (\ref{W-def})
defines a central weight on the dual of $G$ if and only if
\begin{align}\label{Eq:weight-central}
		\om(\sigma) \le \om(\pi)\om(\pi')
    \end{align}\label{co-multi}
for all $\sigma,\pi,\pi' \in \widehat{G}$ with $\sigma \in\, \supp \,\pi\otimes \pi'$.
\end{thm}

\begin{proof}
Following the construction of $W$, it is routine to verify that $W$ is a closed densely
defined positive operator on $H=\bigoplus_{\pi \in \widehat{G}} \ell^2_{d_\pi}$
affiliated to $VN(G)$. Also
$W$ has the inverse $$W^{-1} = \bigoplus_{\pi \in \widehat{G}}\om(\pi)^{-1}1_{M_{d_\pi}}\in VN(G),$$
since $\om$ is bounded away from zero.
Moreover, \eqref{Eq:co-multi} implies that
	$$\Gamma(E_F) =
	\bigoplus_{\sigma \in F} \bigoplus_{\substack{\pi,\pi'\in \widehat{G} \\ \sigma \in\, \supp\, \pi\otimes \pi'}} 1_{M_{d_\sigma}}.$$
Thus it is clear that $\D' = \bigcup_{F\in \F}\Gamma(E_F)(H\otimes_2 H)$ is dense in $H\otimes_2 H$,
so that we can apply Lemma \ref{extension} to define $\Gamma(W)$.
Note that we have
	$$\Gamma(W)\Gamma(E_F)
	= \bigoplus_{\sigma \in F} \bigoplus_{\substack{\pi,\pi'\in \widehat{G} \\ \sigma \in\, \supp\, \pi\otimes \pi'}}
	w(\sigma)1_{M_{d_\sigma}}.$$
On the other hand,
	\begin{align*}
		W^{-1}\otimes W^{-1} &= \bigoplus_{\pi, \pi' \in \widehat{G}}\om(\pi)^{-1}\om(\pi')^{-1}1_{M_{d_\pi}}\otimes 1_{M_{d_{\pi'}}}\\
		&= \bigoplus_{\sigma \in \widehat{G}} \bigoplus_{\substack{\pi,\pi'\in \widehat{G} \\ \sigma \in\, \supp\, \pi\otimes \pi'}} \om(\pi)^{-1}\om(\pi')^{-1} 1_{M_{d_\sigma}},
		\nonumber
	\end{align*}
and so the condition (2) of Definition \ref{Def-weight} is clearly satisfied.
Moreover,
\begin{align}\label{Eq:bdd co prod-compact}
		\Gamma(W) (W^{-1}\otimes W^{-1})
		= \bigoplus_{\sigma \in \widehat{G}} \bigoplus_{\substack{\pi,\pi'\in \widehat{G} \\ \sigma \in\, \supp\, \pi\otimes \pi'}} \om(\sigma) \om(\pi)^{-1}\om(\pi')^{-1} 1_{M_{d_\sigma}}.
	\end{align}
Hence the condition (4) of Definition \ref{Def-weight} is equivalent to
the relation (\ref{Eq:weight-central}). Finally, it is clear that $W_F\in VN(G)'$
for every finite subset $F$ of $\G$, and so, by Remark \ref{R:weighted spaces}(2),
the condition (5) of Definition \ref{Def-weight} is satisfied.
Consequently, $W$ is a central weight on the dual of $G$ if and only if (\ref{Eq:weight-central})
is satisfied.
\end{proof}

The preceding theorem leads us to the following definition. This idea was also considered by
J. Ludwig, N. Spronk, and L. Turowska \cite{LST}.

\begin{defn}\label{D:weight-non abelian-central}
Let $G$ be a compact group, and $W= \bigoplus_{\pi\in \G}\om(\pi)1_{M_{d_\pi}}$ be a central weight on the dual of $G$
for a function $\om : \widehat{G} \rightarrow (\delta,\infty)$ ($\delta >0$) satisfying \eqref{Eq:weight-central}.
For convenience, we use $\om$ to represent $W$, $A(G,\om)$ to represent $A(G,W)$,
$VN(G,\om^{-1})$ to represent $VN(G,W^{-1})$, $C^*_r(G,\om^{-1})$ to represent
$C^*_r(G,W^{-1})$, and use the terminology that $\om$ is a {\bf central weight} on $\G$.
Finally, we define the {\bf symmetrization} of $\om$, denoted by $\Om$, to be
	$$\Om(\pi)=\om(\pi)\om(\overline{\pi}) \ \ \ (\pi \in \G).$$
In particular, we have the completely isometric identification
$$A(G,W)=C^*_r(G,W^{-1})^*.$$

\end{defn}

	\begin{rem}\label{rem-central}

(1) Since $A(G,\om) \subseteq A(G)$ boundedly, we can understand each element in $A(G,\om)$ as a continuous function on $G$.
			More precisely, we have
				$$A(G,\om) \cong \{ f\in C(G) : \|f\|_{A(G, \om)}=\sum_{\pi \in \G} d_\pi  \om(\pi) \|
				\widehat{f}(\pi) \|_1 <\infty\}.$$

(2) It is easy to verify that the symmetrization $\Om$ of $\om$ is also a central weight on $\G$.
It is also easy to check that for any two central weights $\om_1$ and $\om_2$ on $\G$,
the function $\om_1\om_2$ defined by
	$$(\om_1\om_2)(\pi) = \om_1(\pi)\om_2(\pi),\; \pi\in \G$$
is again a central weight on $\G$.

(3) Let $\{G_i\}_{i\in I}$ be a family of compact groups, and
$F(I)$ be the set of finite subsets of $I$.
It follows from \cite[Theorem 27.43 ]{HR2} that the dual of
$\prod_{i\in I}G_i$ consist of all the representations
$$(\pi_i)_{i\in F} : \prod_{i\in F} G_i \to B(\otimes_{i\in F} H_{\pi_i}) \ , \
(x_i)_{i\in F} \mapsto \otimes_{i\in F}\pi_i(x_i) \ (F\in F(I)).$$
Now suppose that, for every $i\in I$, $\om_i$ is a central weight on $\G_i$.
Then it is straightforward to see that
			the product function $\prod_{i\in I}\om_i$ given by
				$$(\prod_{i\in I}\om_i)((\pi_i)_{i\in F}) = \prod_{i\in F}\om_i(\pi_i)$$
			is again a central weight on the dual of $\prod_{i\in I}G_i$
			provided that $\prod_{i\in I}\om_i$ is bounded away from zero as well.

	\end{rem}
Let $m\in \N$, and  $\T^{(m)}$ denotes the $m$-times Cartesian product of $\T$. There are various classical
weight associated to $\widehat{\T^{(m)}}=\Z^{(m)}$ such as
	$$n \mapsto (1+\ln (1+\norm{n}))^a, \; \; n \mapsto (1+\norm{n})^a \ \ \ (a>0),$$
where $\norm{n}$ is the natural norm on $\Z^{(m)}$.
Since the dual of a non-abelian compact group $G$ is not a group anymore, we can not
use this idea to define weights on $\G$. However, as we see in the following example, our
generalization allows us to define very natural weights on $\G$.

\begin{exm}\label{E:weight-compact}
Let $G$ be a compact group, and let $a \geq 0$. We define the functions
$\sg_a$ and $\om_a$ from $\G$ into $[1,\infty)$ by
\begin{align}\label{Eq:weight-non-abelian-ln}
\sg_a(\pi)=(1+\ln d_\pi)^a \ \ \ (\pi \in \G),
\end{align}
\begin{align}\label{Eq:weight-non-abelian-dim}
\om_a(\pi)=d_\pi^a \ \ \ (\pi \in \G).
\end{align}
It follows from the tensor formula (\ref{Eq:tensor formula}) that both $\sg_a$ and $\om_a$
satisfy (\ref{Eq:weight-central}), and so, they are central weights on $\G$. Since the irreducible representations of abelian
groups are 1-dimensional, the preceding weights are trivial if $G$ is abelian.
Thus they are interesting for compact non-abelian groups.
\end{exm}

\subsection{Central weights on the dual of the Heisenberg groups}\label{S:cen weight-Heisenberg group}

	Let $H_d$ $(d\ge 1)$ be the Heisenberg group on $\mathbb{C}^d \times \mathbb{R}$.
	Our references for the Heisenberg groups are \cite[Chapter 1]{Tha} and \cite[Examples 6.7 and 7.6]{Fol}.
	For $h\in \Realzero (= \R\setminus \{0\})$ we consider the Schr\"{o}dinger representations of $H_d$
	acting on $\Hi = L^2(\mathbb{R}^d)$ defined by
		$$\pi_h(z,t)\varphi(\xi) = e^{iht}e^{ih(x\cdot \xi + \frac{1}{2}x\cdot y)}\varphi(\xi +y),$$
	where $\cdot$ is the usual inner product in $\mathbb{R}^d$, $z = x + iy$, $x,y \in \Real^d$ and $\varphi \in \Hi$.
	The Haar measure on $H_d$ is just the Lebesgue measure on $\mathbb{C}^d \times \mathbb{R}$, which will be denoted by $dzdt$.
	The Fourier transform on $H_d$ is defined as follows:
		$$\widehat{f}^{H_d}(h) = \int_{H_d} f(z,t)\overline{\pi_h}(z,t)dzdt,\;\; h\in \Realzero$$
	for $f\in L^1(H_d)$, and the Plancherel theorem says
		$$\int_{\Realzero} \norm{\widehat{f}^{H_d}(h)}^2_{S_2(\Hi)}d\mu(h)
		= \int_{H_d} \abs{f(z,t)}^2 dzdt,$$
	where $S_2(\Hi)$ is the Hilbert-Schmidt class on $\Hi$,
	$d\mu(h) = \frac{\abs{h}^d}{(2\pi)^{d+1}}dh$ on $\Realzero$ and $f\in L^1(H_d)\cap L^2(H_d)$.
	Moreover, it is well known that
		\begin{equation}\label{rep}
		\lambda \stackrel{\text{unitarily}}{\cong} \int^\oplus \overline{\pi_h} d\mu(h),
		\end{equation}
	where $\lambda$ is the left regular representation of $H_d$, and
		$$VN(H_d) \cong L^\infty(\mu;B(\Hi)) \subseteq B(H),$$
	where $H = L^2(\mu; S_2(\Hi))$.
	Note that the above vector-valued $L^\infty$ space $L^\infty(\mu;B(\Hi))$ can be naturally identified with
	the von Neumann algebra tensor product $L^\infty(\mu)\bar{\otimes}B(\Hi)$.
	
	Lastly, we recall the Fourier inverse transform
		$$\F^{-1} : L^1(\mu;S_1(\Hi)) \rightarrow L^\infty(H_d), \; X = (X(h)) \mapsto \F^{-1}(X),$$
	where
		$$\F^{-1}(X)(z,t) = \int_{\Realzero} \text{tr}(\pi_h(z,t)X(h))d\mu(h),$$
	$S_1(\Hi)$ is the trace class on $\Hi$ and $L^1(\mu;S_1(\Hi))$ refers to a vector-valued $L^1$ space.
	
	As in the compact group case, we need to know how the co-multiplication is translated in this setting.
In order to achieve this, we first need the following lemma.

We fix an orthonormal basis $(\xi_i)_{i\ge 1}$ for $\Hi$,
	and let $P_{ji} \in B(\Hi)$ is the operator defined by
	\begin{align}\label{Eq:Orth proj-Heisen}
	P_{ji}(\eta) = \la \eta, \xi_i \ra \xi_j \  ,  \ i,j\ge 1.
	\end{align}
	Then we get the following substitute for Schur orthogonality.
	
	\begin{lem}\label{Lem-Schur-Heisenberg}
	Let $g$ be a function in $\s(\Real)$, the Schwarz class on $\Real$.
	\begin{align}\label{orthogonality}
		\lefteqn{g(h')\delta_{ik}\delta_{jl}}\nonumber \\
		& = \int_{\Realzero} \! \int_{H_d} g(h)\la \pi_h(z,t)\xi_j, \xi_i \ra
		\overline{\la \pi_{h'}(z,t)\xi_k, \xi_l \ra} dzdtd\mu(h),\;\; i,j,k,l\ge 1
		\end{align}
	for every $h'\in \Realzero$.
	\end{lem}
	\begin{proof}
	Let $X = g \otimes P_{ji}$, $i,j\ge 1$. If we take Fourier inverse transform, then we get
		\begin{align*}
		f(z,t) & = \F^{-1}(X)(z,t) = \int_{\Real^*} g(h) tr(\pi_h(z,t) P_{ji})d\mu(h)\\
		& = \int_{\Real^*} g(h)\la \pi_h(z,t)\xi_j, \xi_i \ra d\mu(h), \;\; (z,t)\in H_d.
		\end{align*}
	From the inversion theorem (\cite[theorem 1.3.2]{Tha}) we recover $X$ as the Fourier transform of $f$, so that we have
		\begin{align*}
		X(h') & = g(h')P_{ji}\\
		& = \int_{\Real^*} \! \int_{H_d} g(h)\la \pi_h(z,t)\xi_j, \xi_i \ra \overline{\pi_{h'}(z,t)}dzdtd\mu(h)
		\end{align*}
	for every $h'\in \Realzero$.
	Since $\la P_{ji} \xi_k, \xi_l\ra = \la \xi_k, \xi_i \ra \la \xi_j, \xi_l \ra$, we get the conclusion.
	\end{proof}

	Now let $W = w \otimes id_n$ for some strictly positive $w\in \s(\Real)$ and $n\ge 1$,
	where $id_n = \sum^n_{i=1} P_{ii}$, the $n\times n$ upper-left corner of $1_{B(\Hi)}$.
	We set
		$$f(z,t) = \F^{-1}(W)(z,t)
		= \int_{\Real^*} w(h)\chi^n_{\pi_h}(z,t) d\mu(h),\;\; (z,t)\in H_d,$$
	where
\begin{align}\label{Eq:Character-Heiesn-upper corner}
\chi^n_{\pi_h}(z,t) = \sum^n_{i=1}\la \pi_h(z,t)\xi_i, \xi_i \ra.
\end{align}
	Then we have $\Gamma(W) = \int_{H_d} f(z,t) \lambda(z,t)\otimes \lambda(z,t) dzdt$,
	and if we focus on a particular point $(h', h'')\in \Real^* \times \Real^*$, then by \eqref{rep} we have
	
	\begin{align*}
		\Gamma(W)(h',h'')
		& = \int_{H_d} f(z,t)\overline{\pi_{h'}(z,t)}\otimes \overline{\pi_{h''}(z,t)} dzdt\\
		& = \int_{\Real^*} \! \int_{H_d} w(h)\chi^n_{\pi_h}(z,t)
		\overline{\pi_{h'}(z,t)}\otimes \overline{\pi_{h''}(z,t)} dzdt
	\end{align*}
	for almost every $(h', h'')\in \Real^* \times \Real^*$.
	By the Stone-von Neumann theorem (\cite[Theorem 6.49]{Fol}) for $h'+h'' \neq 0$
	(note that the cases $h'+h'' =0$ are measure zero with respect to $\mu \times \mu$) we have
		$$\pi_{h'}\otimes \pi_{h''} \cong \bigoplus_\alpha \pi^\alpha_{h'+h''},$$	
	where $\pi^\alpha_{h'+h''}$ are copies of $\pi_{h'+h''}$. Thus, by \eqref{orthogonality}
and \eqref{Eq:Character-Heiesn-upper corner} we have
		\begin{align*}
		\Gamma(W)(h', h'')
		& = \bigoplus_\alpha \! \int_{\Real^*}\! \int_{H_d} w(h)\chi^n_{\pi_h}(z,t)
		\overline{\pi^\alpha_{h'+h''}(z,t)} dzdt d\mu(h) \\
		& = \bigoplus_\alpha w(h'+h'')id_n
		\end{align*}
	for almost every $(h', h'')\in \Real^* \times \Real^*$.
Note that the above equality can be extended to any $\om \in L^\infty(\R)$
since $\s(\R)$ is $w^*$-dense in $L^\infty(\R)$,
and we can replace $id_n$ by $1_{B(\Hi)}$ by $w^*$-$w^*$ continuity.
Thus, for any $\om \in L^\infty(\Real)$ and $W = \om \otimes 1_{B(\Hi)}$ we have
	\begin{equation}\label{co-multi2}
	\Gamma(W)(h', h'') = w(h'+h'')1_{B(\Hi)}\otimes 1_{B(\Hi)}
	\end{equation}
for almost every $(h', h'')\in \Real^* \times \Real^*$.
Note that we used the fact that $\bigoplus_\alpha 1_{B(\Hi)}$ is identified with $1_{B(\Hi)}\otimes 1_{B(\Hi)}$.
	
Now we consider a continuous positive function $\om$ on $\R$ which is bounded away from zero.
We would like to construct a central weight associated to $\om$.
For $m \in \N$, we consider the projection $E_m$ in $VN(H_d)$ given by
	$$E_m = 1_{[-m,m]}\otimes 1_{B(\Hi)}.$$
It is clear that $(E_m)_{m\ge 1}$ is an increasing net of projections in $VN(H_d)$
and $\D = \bigcup_{m\ge 1}E_m(H)$ is dense in $H = L^2(\mu;S_2(\Hi))$.
Let $W_m$ be the operator in $VN(H_d)$ given by
	$$W_m := (\om 1_{[-m,m]}) \otimes 1_{B(\Hi)}.$$
Consider a linear operator $W_0$ with the domain $\D$ defined by
	$$W_0(h) := W_m(h),\;\; h\in E_m(H).$$
If we apply the same argument as in \cite[Lemma 5.6.1]{KR83}, then we can show that
$W_0$ is closable with the self-adjoint closure. We will denote this closure by
	\begin{equation}\label{W-def-Heisenberg}
	W = \om \otimes 1_{B(\Hi)}.
	\end{equation}
Similar to the compact groups, we can exactly determine when $W$ defines
a weight on the dual of $G$.

\begin{thm}\label{T:central weight-Heisen}
Let $w : \R \rightarrow (\delta,\infty)$
be a continuous function, where $\delta>0$. The operator $W$ constructed in (\ref{W-def-Heisenberg})
defines a central weight on the dual of the Heisenberg group $H_d$ if and only if
\begin{equation}\label{Eq:weight-Heisenberg-central}
	w(h'+h'')\le w(h')w(h'')
	\end{equation}
	for every $h'$ and $h''$ in $\R$.
\end{thm}

\begin{proof}
Following the construction of $W$, it is routine to verify that $W$ is a closed densely
defined positive operator on $H_d$ affiliated to $VN(H_d)$. Also
$W$ has the bounded inverse
$$W^{-1} = \om^{-1} \otimes 1_{B(\Hi)},$$
since $\om$ is bounded away from zero.
Moreover, \eqref{co-multi2} implies that
	$$\Gamma(E_m)(h',h'')	= 1_{[-m,m]}(h'+h'')1_{B(\Hi)} \otimes 1_{B(\Hi)}$$
for almost every $(h', h'')\in \Real^* \times \Real^*$.
Thus it is clear that
	$$\D' = \bigcup_{m\ge 1}\Gamma(E_m)(H\otimes_2 H)$$
is dense in $H\otimes_2 H$,
so that we can apply Lemma \ref{extension} to define $\Gamma(W)$. Note that
	$$\Gamma(W)(h',h'') = 	1_{[-m,m]}(h'+h'')w(h'+h'')1_{B(\Hi)} \otimes 1_{B(\Hi)}$$
on $\Gamma(E_m)(H\otimes_2 H)$, for every $m \ge 1$ and almost every $(h', h'')\in \Real^* \times \Real^*$.
On the other hand,
	$$(W^{-1} \otimes W^{-1})(h',h'') = w^{-1}(h')w^{-1}(h'') 1_{B(\Hi)} \otimes 1_{B(\Hi)}$$
for every $(h', h'')\in \Real^* \times \Real^*$.
Then clearly $\D'$ and $\D_0 = \{x\in H\otimes_2 H : (W^{-1}\otimes W^{-1})x \in \D'\}$ both contain $C_{00}(\R^* \times \R^*) \otimes S_2(\Hi \otimes \Hi)$,
where $C_{00}(\R^* \times \R^*)$ refers to the space of continuous functions on $\Real^*\times \R^*$ with compact support.
Moreover, $W^{-1} \otimes W^{-1}$ preserves $C_{00}(\R^* \times \R^*) \otimes S_2(\Hi \otimes \Hi)$,
and since $C_{00}(\R^* \times \R^*) \otimes S_2(\Hi \otimes \Hi)$ is dense in $H\otimes_2 H$, the condition (2) of Definition \ref{Def-weight} is satisfied.
Moreover, the condition (4) of Definition \ref{Def-weight} is equal to
	\begin{equation*}
	w(h'+h'')\le w(h')w(h'')
	\end{equation*}
	for almost every $h'$ and $h''$ in $\R ^*$ which is equivalent to the relation
(\ref{Eq:weight-Heisenberg-central}) since $w$ is continuous.
Finally, it is clear that $W_m\in VN(G)'$
for every $m\in \N$, and so, by Remark \ref{R:weighted spaces}(2),
the condition (5) of Definition \ref{Def-weight} is satisfied.
Consequently, $W = \om \otimes 1_{B(\Hi)}$ defines a weight on the dual of $H_d$
if and only if (\ref{Eq:weight-Heisenberg-central}) holds.
\end{proof}

The preceding theorem is the motivation behind the following definition.

\begin{defn}\label{D:weight-Heisenberg-central}
Let $W = \om \otimes 1_{B(\Hi)}$ be a central weight on the dual of $H_d$
for a continuous function $\om : \Real \rightarrow (\delta,\infty)$ $(\delta>0$) satisfying \eqref{Eq:weight-Heisenberg-central}.
For convenience, we use $\om$ to represent $W$, $A(H_d,\om)$ to represent $A(H_d,W)$,
and use the terminology that $\om$ is a {\bf central weight} on $\widehat{H_d}$.
\end{defn}

\begin{exm}\label{E:weight-Heisenberg}
Let $a \geq 0$. We define the function $\tau_a : \Real \rightarrow [1,\infty)$ by
	\begin{align}\label{Eq:weight-Heisenberg-dim}
	\tau_a(x)= (1+\abs{x})^a \ \ \ (x\in \Real).
	\end{align}
By Theorem \ref{T:central weight-Heisen} and Definition \ref{D:weight-Heisenberg-central},
$\tau_a$ is a central weight on $\widehat{H_d}$.
\end{exm}

\section{Compact groups}\label{S:BF alg-compact}

Throughout this section, $G$ is always assumed to be a compact group.
We start with showing certain functorial property that holds for the Beurling-Fourier
algebras.

\subsection{Functorial Property}\label{S:Functorial Property}

Consider the map $Q :A(G\times G) \to
C(G)$ defined by
\[
Qw(s)=\int_G w(sr,r)dr.
\]

We denote the image of $Q$ by $A_\Del(G)$. We endow
$A_\Delta(G)$ with the operator space structure which makes $Q$ a complete quotient map.
We also note that
\[
N:A_\Delta(G) \rightarrow A(G\cross G), \quad
Nu(s,t)=u(st^{-1})
\]
is a complete isometry. As in \cite[Theorem 2.6]{FSS1}, if
we repeat the procedure above we obtain
\begin{equation}\label{E:A-DEl-formula}
A_{\Del^{n+1}}(G)=Q(A_{\Del^n}(G\times G)) \ \ \ (n\in \N).
\end{equation}
We can do a similar construction with
\[
\check{Q}:A(G\times G) \rightarrow C(G) , \quad
\check{Q} w(s)=\int_G w(st,t^{-1})dt.
\]
We denote the image of $\check{Q}$ by $A_\gamma(G)$. We endow
$A_\gamma(G)$ with the operator space structure which makes $\check{Q}$ a
complete quotient map.  We also note that
\[
\check{N}:A_\gamma(G) \rightarrow A(G\cross G:\check{\Del}) ,
\quad \check{N}u(s,t)=u(st)
\]
is a complete isometry.  If
we repeat the procedure above we obtain
\begin{equation}\label{Eq:A-DEl-formula}
A_{\gamma^{n+1}}(G)=\check{Q}(A_{\gamma^n}(G\times G)) \ \ \ (n\in \N\cup \{0\}).
\end{equation}
It follows immediately that, for each $n\in \N$, $A_{\gamma^n}(G)$ is a
closed unital subalgebra of the Fourier algebra $A(G^{(2n)})$. Moreover, by \cite[Theorem 4.1]{FSS1},
\begin{equation}\label{Eq:norm-gamma}
\|f\|_{A_{\gamma^n}(G)} = \sum_{\pi \in \G} d_\pi^{2^{n}+1} \norm{ \widehat{f}(\pi)}_1.
\end{equation}

Let $\om$ be a central weight on $\widehat{G}$, and let $\Om$
be the symmetrization of $\om$ (Definition \ref{D:weight-non abelian-central}).
Since  $A(G\times G, \om \times \om)$ is a subalgebra of $A(G\times G)$, we can restrict
the map $Q$ to  $A(G\times G, \om \times \om)$. We denote
$$A_\Del(G, \Om)=Q(A(G\times G, \om \times \om))$$
and endow $A_\Del(G, \Om)$ with the operator space structure so that
it became a complete quotient of $A(G\times G, \om \times \om)$. It is clear that
$A_\Del(G, \Om)$ is a completely contractive Banach algebra.
Moreover
\[
N:A_\Delta(G, \Om) \rightarrow A(G\cross G, \om \times \om), \quad
Nu(s,t)=u(st^{-1})
\]
induces a completely isometric algebraic monomorphism from
$A_\Del(G, \Om)$ into  $A(G\times G, \om \times \om)$.

The following theorem explains the motivation behind using $\Om$
in the preceding definition.

\begin{thm}
Let $G$ be a compact group, and let $\om$ be a central weight on $\G$. Then
	$$A_\Del(G, \Om)
	= \{f\in C(G) : \sum_{\pi \in \G} d_\pi^{3/2}  \Om(\pi) \norm{\widehat{f}(\pi)}_2< \infty \} .$$
Moreover, for every $f\in A_\Del(G, \Om)$, we have:
\begin{eqnarray*}
\|f\|_{A_\Del(G, \Om)} &=& \inf \{\|u\|_{\om \times \om} : u\in A(G\times G, \om \times \om), Q u=f \}
\\ &=& \sum_{\pi \in \G} d_\pi^{3/2}  \Om(\pi) \| \widehat{f}(\pi) \|_2.
\end{eqnarray*}
\end{thm}

\begin{proof}
It suffices to show that, for $f\in C(G)$,
$$f\in A_\Del(G, \Om) \ \ \text{iff} \ \ \ \sum_{\pi \in \G} d_\pi^{3/2}  \Om(\pi) \| \widehat{f}(\pi) \|_2< \infty.$$
We note that $f \in A_\Del(G, \Om)$ if and only if $Nf \in A(G\times G, \om \times \om)$, in which
case $\|f\|_{A_\Del(G, \Om)}=\|Nf\|_{\om \times \om}.$ On the other hand,
following a similar argument as in the proof of Theorem 2.2 in \cite{FSS1},
we can compute $\|Nf \|_{\om\times \om}$,
which is $\sum_{\pi \in \G} d_\pi^{3/2}  \Om(\pi) \| \widehat{f}(\pi) \|_2$.
This completes the proof.
\end{proof}

We collect some notations for ideals which we will need in this section.
	\begin{defn}
	Let $E$ be a closed subset of $G$. We define
		$$E^* := \{(s,t)\in G\times G : st^{-1}\in E \}.$$
	For any central weight $\om$ on $\G$ we define the ideal $I_\om(E)$ to be
	the $\|\cdot \|_{A(G,\om)}$-closure of $\{f\in \fT(G) : f=0 \ \text{on}\ E\}$. Similarly,
		$$I_{\Del, \Om}(E)=\{f\in A_\Del(G, \Om) : f=0 \ \text{on}\  E\}$$
	and
		$$I_{\om \times \om}(E^*) = \{g\in A_\Del(G\times G, \om \times \om) : g=0 \ \text{on}\ E^*\}.$$
	\end{defn}

\begin{thm}\label{T:Proj-weight}
Let $G$ be a compact group, and let $\om$ be a central weight on $\G$. If $E$ is a closed
subset of $G$, then we have\\
$($i$)$ $Q I_{\om \times \om}(E^*)=I_{\Del, \Om}(E)$.\\
$($ii$)$ $I_{\om \times \om}(E^*)$ is the closed ideal generated by $N I_{\Del, \Om}(E)$.
\end{thm}
\begin{proof}
Since $\om$ is bounded away from zero, $A(G\times G, \om \times \om)$ satisfies the
assumption of \cite[Theorem 1.4]{FSS1}. Thus it is a special case of \cite[Theorem 1.4]{FSS1}.
\end{proof}

The following proposition is shown to hold in \cite[Theorem 3.7.13]{RS} when $G$ is abelian.
We prove it for the general case with a different method
and later use it to relate the properties of different Beurling-Fourier algebras together.
But first we need the following definition.

\begin{defn}\label{D:weight-Max}
For a closed subgroup $H$ of $G$, we say that $\pi \in \G$ is an {\bf extension} of
$\tau \in \widehat{H}$ if $\tau \in \supp\, \pi_{|_H}$, where $\pi_{|_H}$ is the
representation on $H$ obtained by restricting on $H$.
\end{defn}

\begin{prop}\label{P:weight-subgroup}
Let $\om$ be a central weight on $\G$, and let $H$ be a closed subgroup of $G$. Define
the function $\om_H : \hH \to (0,\infty)$ by
$$\om_H(\pi)=\inf \{\om(\widetilde{\pi}) \mid \widetilde{\pi} \ \text{is an extenstion of}\ \pi \}.$$
Then:\\
$(i)$ $\om_H$ is a central weight on $\hH$.\\
$(ii)$ The restriction map $R_H: \fT(G) \to \fT(H)$ extends to a complete quotint map from $A(G,\om)$ onto
$A(H,\om_H)$.
\end{prop}

\begin{proof}
(i) By \cite[27.46]{HR2}, $\om_H$ is well-defined. We will show that $\om_H$ is a weight
on $\hH$. Let $\pi, \rho \in \hH$ and $\tau\in \hH$ such that $\tau \in \supp\, (\pi\otimes \rho)$.
We want to prove that
$$\om_H(\tau) \leq \om_H(\pi)\om_H(\rho),$$
or equivalently,
\begin{equation}\label{Eq:tensor-restriction}
\om_H(\tau) \leq \om(\widetilde{\pi})\om(\widetilde{\rho}),
\end{equation}
for every extension $\widetilde{\pi}$ and $\widetilde{\rho}$ of $\pi$ and $\rho$, respectively.
Let
\begin{equation}\label{Eq:tensor-tilda}
\widetilde{\pi}\otimes \widetilde{\rho}=\bigoplus_{i=1}^n \sg_i,
\end{equation}
where $\sg_i\in \G$ (note that $\sg_i$'s may not be all distinct). We claim that, for some $i_0$, $\sg_{i_0}$ is an extenstion
of $\tau$. To see this, first note that
$$\widetilde{\pi}_{|_H}=\bigoplus_{j=1}^m \pi_j  \ , \  \widetilde{\rho}_{|_H}=\bigoplus_{k=1}^p \rho_k,$$
with $\pi_i, \rho_i \in \hH$ with $\pi_1=\pi$, and $\rho_1=\rho$. Thus
$$(\widetilde{\pi} \otimes \widetilde{\rho})_{|_H}=\Big( \pi \otimes \rho\Big)
\oplus \Big(\bigoplus_{j+k>2}^{m+p} \pi_j \otimes \rho_k\Big).$$
Since, by our assumption, $\tau \in \supp\, (\pi \otimes \rho)$, $\tau$ appears in the decomposition of
$(\widetilde{\pi} \otimes \widetilde{\rho})_{|_H}$ into the irreducible elements of $\hH$. Hence,
by (\ref{Eq:tensor-tilda}), $\tau$ appears in the decomposition of $\bigoplus_{i=1}^n {\sg_i}_{|_H}.$
Therefore, by Schur orthogonality relation, for some $i_0$, $\sg_{i_0}$ is an extension
of $\tau$. Thus
$$\om_H(\tau)\leq \om(\sg_{i_0}) \leq \om(\widetilde{\pi})\om(\widetilde{\rho}),$$
which proves (\ref{Eq:tensor-restriction}). This completes the proof.\\
(ii) Let $\iota : C^*_r(H) \to C^*_r(G)$ be the $*$-isomorphism defined by
$$\iota(L_f)=\int_H f(h) \lambda_G(h) dh, \;\; f\in L^1(H),$$
where $L_f$ is the convolution operator by $f$ on $L^2(H)$.
It is easy to check that $R_H|_{\fT(H)} = \iota^*|_{\fT(H)}$,
so that it suffices to show that $\iota$ extends to a complete isometry from $C^*_r(H,\om_H^{-1})$
into $C^*_r(G,\om^{-1})$.

Now we fix $n\ge 1$ and consider a finite sequence of distinct representations
$(\sigma_k)^N_{k=1}\subseteq \widehat{H}$.
Let $A$ be an element in $M_n(VN(H, \om^{-1}_H))$ supported on $(\sigma_k)^N_{k=1}$, i.e.
	$$A = \sum^n_{i,j=1} e_{ij}\otimes \Big[ \bigoplus^N_{k=1} A^k_{ij} \Big],$$
where $A^k_{ij} \in M_{d_{\sigma_k}}$. From (\ref{Eq:Fourier trans-compact-trig poly}), (\ref{Eq:Inverse Fourier trans-compact-trig poly}), and (\ref{Eq:Fourier trans-compact-VN(G)}) it follows that
for every $1\leq i,j \leq n$, there is $f_{ij}\in L^1(H)$ such that
$$\F(L_{f_{ij}})=\bigoplus^N_{k=1} A^k_{ij}.$$
Moreover, for every $h\in H$,
	$$f_{ij}(h) = \sum^N_{k=1}d_{\sigma_k}\text{tr}(A^k_{ij}\sigma_k(h)).$$
Hence we have
\begin{align*}
		\iota_n([L_{f_{ij}}]) & = \Big [\int_H f_{ij}(h)\lambda_G(h)dh \Big ]\\
		& = \sum^n_{i,j=1} e_{ij}\otimes
		\Big[ \bigoplus_{\pi\in \widehat{G}}\int_H
		\sum^N_{k=1}d_{\sigma_k}\text{tr}(A^k_{ij}\sigma_k(h)) \overline{\pi(h)}dh \Big].
	\end{align*}
The integrals in the above formula are zero unless there is some $\sigma_k \in$ supp$\,\pi$, equivalently,
$\pi$ extends some $\sigma_k$. Thus
	\begin{align*}
		\lefteqn{\norm{\iota_n([L_{f_{ij}}])}}\\
		& = \sup_{\pi\in\widehat{G}}\Big\{ \frac{1}{\om(\pi)}
		\norm{\sum^n_{i,j=1} e_{ij}\otimes \int_H \sum^N_{k=1}d_{\sigma_k}\text{tr}(A^k_{ij}\sigma_k(h))
		\overline{\pi(h)}dh}_{M_n(M_{d_\pi})} \Big\}\\
		& = \sup_{\sigma_k\in\text{supp}\, \pi}\Big\{ \frac{1}{\om(\pi)}
		\norm{\sum^n_{i,j=1} e_{ij}\otimes \int_H \sum^N_{k=1}d_{\sigma_k}\text{tr}(A^k_{ij}\sigma_k(h))
		\overline{\sigma_k(h)}dh}_{M_n(M_{d_\pi})} \Big\}\\
		& = \max_{1\le k\le N}
		\frac{1}{\inf\{\om(\pi):\sigma_k\in\text{supp}\, \pi\}}
		\norm{\sum^n_{i,j=1} e_{ij}\otimes B^k_{ij}}_{M_n(M_{d_\pi})}\\
		& = \max_{1\le k\le N}
		\frac{\norm{\sum^n_{i,j=1} e_{ij}\otimes B^k_{ij}}_{M_n(M_{d_\pi})}}
		{\om^{-1}_H(\sigma_k)}\\
		& = \norm{L_f}_{M_n(C^*_r(H,\om^{-1}_H))},
	\end{align*}
where $B^k_{ij} = \int_H \sum^N_{k=1}d_{\sigma_k}\text{tr}(A^k_{ij}\sigma_k(h)) \overline{\sigma_k(h)}dh$.
By a standard density argument $\iota$ extends to a complete isometry from $C^*_r(H,\om_H^{-1})$
into $C^*_r(G,\om^{-1})$.

\end{proof}

We note that every (infinite) compact group contains (infinite) abelian subgroups \cite{Z}.
Thus by the preceding proposition, every Beurling-Fourier algebra has certain classical Beurling algebras as complete quotients.
This can be very useful particularly when the abelian subgroups can be chosen so that
they contain various information about the original compact group.
This happens, for example, in the case where $G$ is a compact Lie group and $H$ is any maximal torus of $G$.
We will show in details in Section \ref{S:weight-SU(2)}
how this idea can be applied to relate properties of Beurling-Fourier algebra on $SU(2)$ and Beurling algebras on $\T$.

\subsection{Operator Amenability}\label{S:Operator amen-BF alg}

In this section, we present certain criteria for investigating the
operator amenability of $A(G,\om)$. We will later show that this criteria can be applied to
large classes of weights. But first, we need to recall the following terminologies:

We recall that a completely contractive Banach algebra $A$ is $K$-operator amenable
if there is a virtual diagonal $M\in (A\widehat{\otimes} A)^{**}$ such that $\|M\|=K$.
The {\it operator amenability constant} of $A$ is the smallest $K$ such that $A$
is $K$-operator amenable.  We also recall that if $A$ is
a Banach algebra of continuous functions on a locally compact
space $X$, then for every $x\in X$, a functional $d\in A^*$ is called a {\it point derivation}
at $x$ if
$$d(ab)=a(x)d(b)+b(x)d(a) \ \ (a,b\in A).$$

\begin{thm}\label{T:Amen-weak Amen-Delta}
Let $G$ be a compact group with the identity $e$, and let $\om$ be a central weight on $\G$. Then:\\
$($i$)$ $A(G,\om)$ is operator amenable if and only if $I_{\Del, \Om}(\{e\})$
has a bounded approximate identity.\\
$($ii$)$ $A(G,\om)$ is $K$-operator amenable if and only if there is $F\in A_\Del(G,\Om)^{**}$
such that $\|F\|=K$, $\la F, \delta_e \ra=1$, and
$$ f\cdot F=f(e)F \ \ (f\in A_\Del(G,\Om)).$$
$($iii$)$ $A(G,\om)$ is operator weakly amenable if and only if $\overline{I_{\Del, \Om}(\{e\})^2}=
I_{\Del, \Om}(\{e\})$ is essential,
or equivalently, there is no non-zero continuous point derivation on $A_\Del(G,\Om)$ at
$e$.
\end{thm}

\begin{proof}
(i) and (iii). If we let $m: A(G,\om) \widehat{\otimes} A(G,\om) \to A(G,\om)$ be the multiplication
map, then it is easy to verify that
$$I_{\om \times \om}(\Del)=\ker m.$$
Thus following the arguments in \cite{Ruan} and \cite{Sp}, we see that
$A(G,\om)$ is operator amenable (respectively, operator
weakly amenable) if and only if $I_{\om \times \om}( \Del )$ has a bounded approximate identity
(respectively, $\overline{I_{\om \times \om}( \Del )^2}=I_{\om \times \om}( \Del )$).
Thus the results follows from Theorem \ref{T:Proj-weight} and the fact that $\{e\}^*=\Del$. \\
(ii) Let $A(G,\om)$ be $K$-operator amenable, and let $M$ be a virtual
diagonal for $A(G,\om)$ with $\|M\|=K$. Let $Q^{**}$ be the second adjoint
of \linebreak $Q : A(G\times G, \om \times \om) \to A_\Del(G,\Om)$ defined
in Section \ref{S:Functorial Property}, and let $F=Q^{**}(M)$.
Then it is routine to verify that $F$ holds the required properties of (ii). Conversely,
if such an $F\in A_\Del(G,\Om)^{**}$ exits, then $M=N^{**}(F)$ is a virtual diagonal
for $A(G,\om)$ with $\|M\|=\|F\|=K$.
\end{proof}

In \cite[Theorem 4.1]{J1}, B. E. Johnson computed the amenability
constant of the Fourier algebra of a finite group. The following
theorem is the quantization of Johnson's result to Beurling-Fourier algebras
on a finite group and its proof is inspired by that of Johnson's.

\begin{thm}\label{T:weight-amen-finite group}
Let $G$ be a finite group, and let $\om$ be a central weight on $\G$. Then $A(G,\om)$ is
operator amenable with the operator amenability constant $\displaystyle \frac{\sum_{\pi\in \G} d_\pi^2\Om(\pi)}{\sum_{\pi\in \G} d^2_\pi}$.
\end{thm}

\begin{proof}
It is straightforward to verify that $\delta_e$, the dirac function at $\{e\}$,
is the unique element in $A_\Del(G,\Om)^{**}=A_\Del(G,\Om)$ satisfying
the assumption of Theorem \ref{T:Amen-weak Amen-Delta}(ii). Thus
it follows from Theorem \ref{T:Amen-weak Amen-Delta}(ii) that $A(G,\om)$ is operator amenable
and the operator amenability constant is the $\|\cdot \|_{A_\Del(G,\Om)}$-norm of $\delta_e$. However  $|G| \widehat{\delta_e}(\pi)=1_{B(H_\pi)}$ for every $\pi\in \G$. Therefore, considering the well-known fact that
$|G|=\sum_{\pi\in \G} d^2_\pi$, we have
$$\| \delta_e \|_{A_\Del(G,\Om)}=\sum_{\pi\in \G} d^{\frac{3}{2}}_\pi \Om(\pi) \|\widehat{\delta_e}(\pi) \|_2
=\frac{\sum_{\pi\in \G} d_\pi^2\Om(\pi)}{\sum_{\pi\in \G} d^2_\pi}.$$
\end{proof}

\begin{cor}\label{C:weight-finite group product-amen}
Let $\{G_i \}_{i\in \N}$ be a family of finite groups, and let, for each $i\in \N$, $\om_i$
be the central weight on $\G_i$ and $\Om_i$ its symmetrization. Let $G=\prod_{i\in I} G_i$ and $\om=\prod_{i\in I} \om_{i}$. Suppose further that $\om$ is bounded away from zero.
Then $A(G,\om)$ is operator amenable if and only if
$$M_{G,\om}:=\prod_{i\in \N} \frac{\sum_{\pi\in \G_i} d_\pi^2\Om_i(\pi)}{\sum_{\pi\in \G_i} d^2_\pi}$$
is convergent. In this case, $M_{G, \om}$ is the operator amenability constant of $A(G,\om)$.
\end{cor}

\begin{proof}
It is clear that, for each $n \in \N$, there is a complete quotient map from $A(G,\om)$ onto $A(\prod_{i=1}^n G_i, \prod_{i=1}^n \om_i)$. Also
it follows from Theorem \ref{T:weight-amen-finite group} that the amenability constant of   $A(\prod_{i=1}^n G_i, \prod_{i=1}^n \om_i)$
is $$\prod_{i=1}^n \frac{\sum_{\pi\in \G_i} d_\pi^2\Om_i(\pi)}{\sum_{\pi\in \G_i} d^2_\pi}.$$ Therefore the operator amenability constant of
$A(G,\om)$ is at least $M_{G,\om}$. In particular, if
$A(G,\om)$ is operator amenable, then  $M_{G,\om}$
is convergent. Conversely, suppose that $M_{G,\om}$
is convergent. Consider the sequence of continuous functions $\{f_n\}$ on $G$ defined by
	$$ f_n(\{x_i \})= \begin{cases} 1 &\text{$x_1=\cdots=x_n=e$}\\ 0 &\text{otherwise.} \end{cases}$$
For each $n$, we see that
$$\|f_n\|_{A_\Del(G,\Om)}=\prod_{i=1}^n \frac{\sum_{\pi\in \G_i} d_\pi^2\Om_i(\pi)}{\sum_{\pi\in \G_i} d^2_\pi}.$$
In particular, $\{f_n\}$ is bounded in $A_\Del(G,\Om)$. Let $F$ be a weak$^*$-cluster point of $\{f_n\}$ in $A_\Del(G,\Om)^{**}$. Then it is straightforward to verify that $F$ satisfies
the hypothesis of Theorem \ref{T:Amen-weak Amen-Delta}(ii). Moreover $\|F\|=M_{G,\om}$. Thus $A(G,\om)$ is operator amenable
with the operator amenability constant $M_{G,\om}$.
\end{proof}

The preceding corollary has an interesting application when each $G_i$ is $S_3$; the permutation group on $\{1,2,3\}$.
It is well-known (e.g. \cite[27.61(a)]{HR2}) that $\widehat{S_3}$
have two 1-dimensional elements and one 2-dimensional element.
Using this fact, we can construct Beurling-Fourier algebras on countably infinite products of $S_3$ so that they are operator amenable.
Moreover, we can let amenability constant be as large as we would like!
This is something that does not happen in the Fourier algebra case since the amenability constant is always 1 \cite{Ruan}.

\begin{thm}
Let $G_i=S_3$ for every $i\in \N$, and let, $\om_{a_i}$ be the central weight (\ref{Eq:weight-non-abelian-ln}) on $\G_i$ defined in Example \ref{E:weight-compact}. Let $G=\prod_{i\in \N} G_i$ and $\om=\prod_{i\in \N} \om_i$. Then:\\
$(i)$ $A(G,\om)$ is operator amenable if and only if
$\sum_{i=1}^\infty (2^{2a_i}-1)$ is convergent;\\
$(ii)$ For every $1\leq K < \infty$, we can choose $\{a_i \}$ so that
the amenability constant of $A(G,\om)$ is $K$.
\end{thm}

\begin{proof}
By Corollary \ref{C:weight-finite group product-amen}, $A(G,\om)$ is operator amenable if and only if its amenability
constant which is
$$\prod_{i\in \N} \frac{\sum_{\pi\in \G_i} d_\pi^2\Om_i(\pi)}{\sum_{\pi\in \G_i} d^2_\pi}
=\prod_{i\in \N} \frac{1+2^{2a_i+1}}{3}=\prod_{i\in \N} (1+ \frac{2^{2a_i+1}-2}{3})$$
is finite. However this happens if and only if $\displaystyle \sum_{i=1}^\infty \frac{2^{2a_i+1}-2}{3}$ is convergent.
Thus (i) holds since $$\sum_{i=1}^\infty \frac{2^{2a_i+1}-2}{3}=2/3 \sum_{i=1}^\infty (2^{2a_i}-1).$$
The proof of (ii) is easy. In fact, there are various way to chose the required $\{a_i \}$. For example,
we can pick $a_1$ so that $K=\displaystyle \frac{1+2^{2a_1+1}}{3}$ and take $a_i=0$ for $i=2,3,\cdots$.
\end{proof}

In \cite{G2}, N. Gr{\o}nb{\ae}k has shown that the Beurling algebra $L^1(H,\om)$ on a
locally compact group $H$ is amenable if and only if $H$ is amenable and $\Om$ is bounded,
where $\Om$ is the symmetrization of $\om$ given by $\Om(x) = \om(x)\om(x^{-1})$, $x\in H$.
In the below we prove a weaker version of Gr{\o}nb{\ae}k's result for Beurling-Foureir algebras on
compact groups. This presents a nice duality to Gr{\o}nb{\ae}k's criteria.

\begin{thm}\label{T:weight-op. amen}
Let $G$ be a compact group, and let $\om$ be a central weight on $\G$. Then the followings
holds:\\
$(i)$ If $\Om$ is bounded, then $A(G,\om)$ is operator amenable;\\
$(ii)$ If $\lim_{\pi \to \infty} \Om(\pi)=\infty$, then $A(G,\om)$ is not
operator amenable.
\end{thm}

\begin{proof}
(i) If $\Om$ is bounded, then $A_\Del(G,\Om)=A_\Del(G)$. However, by
\cite[Theorem 3.9(iii)]{FSS1}, $I_\Del(\{e\})$ has a bounded approximate identity.
Hence the result follows from Theorem \ref{T:Amen-weak Amen-Delta}(i).\\
(ii) Suppose that $A(G,\om)$ is operator amenable. Then by Theorem \ref{T:Amen-weak Amen-Delta}(ii)
and going to an appropriate subnet, there is a bounded net $\{ f_\alpha \}_{\alpha} \subset A_{\Del,\Om}(G)$
such that $f_\alpha(e)=1$ for all $\alpha$, and
\begin{equation}\label{Eq:amen-delta}
ff_\alpha=f(e)f_\alpha\ \text{for all}\ f\in A_\Del(G,\Om).
\end{equation}
Since for every $n\in \N$ and $T\in M_n$, $\|T\|_1 \leq n^{1/2}\|T \|_2$, and so,
we have $A_\Del(G,\Om)\subseteq A(G,\Om)$.
Therefore we can assume that $\{ f_\alpha \}_{\alpha} \subset A(G,\Om)$.
Now let $g$ be a weak$^*$-cluster point
of $\{f_\alpha\}$ in $A(G,\Om)=C^*_r(G, \Om^{-1})^*$. Let $I :\G \to \oplus_{\pi \in \G}M_{d_\pi}$
be defined by $I(\pi)=1_{B(H_\pi)}$ for all $\pi \in \G$.
It is clear that $I$ is the inverse Fourier transform of $\delta_e$, the evaluation functional
on $A(G, \Om)$ at $\{ e\}$.
Now since $\lim_{\pi \to \infty} \Om(\pi)=\infty$, it follows that $\delta_e=\check{I} \in C^*_r(G, \Om^{-1})$, and so,
	$$g(e)=\la g, \delta_e \ra=\lim_{\alpha} \la f_\alpha, \delta_e \ra=1.$$
On the other hand, it follows routinely from (\ref{Eq:amen-delta}) that $g(t)=0$ if $t\neq e$.
Since $g$ is continuous, it follows that $G$ is finite which contradict the fact that
$\lim_{\pi\to \infty} \Om(\pi)=\infty$. This completes the proof.

\end{proof}

\begin{cor}\label{C:weight-Op amen-simple Lie group}
Let $G$ be a compact connected, simple Lie group and let $\sg_a$ and $\om_a$ be the
central weights on $\G$ defined in Example \ref{E:weight-compact}.
Then $A(G,\sg_a)$ or $A(G,\om_a)$ is operator amenable if and only if $a=0$.
\end{cor}

\begin{proof}
When $a=0$, $A(G,\sg_a)=A(G,\om_a)=A(G)$. Hence the result follows from
\cite{Ruan}. For the converse, it is shown in the proof of \cite[Lemma 9.1]{Rid}
that, for any positive integer $n$,
there are only finitely many elements in $\G$ whose dimension is $n$. Thus
$d_\pi\to \infty$ as $\pi \to \infty$. Therefore if $a>0$, by
Theorem \ref{T:weight-op. amen}, neither of $A(G,\sg_a)$ or $A(G,\om_a)$
is operator amenable.
\end{proof}

\subsection{Operator weak amenability}\label{S:BF alg-compact-OWA}

In this section we consider the question of whether a Beurling-Fourier algebra on
a compact group can be operator weakly amenable. Our main tool is to use the criterion
presented in Theorem \ref{T:Amen-weak Amen-Delta}(iii). We first need the following definition.

\begin{defn}\label{D:weight-Max}
Let $\pi$ be a continuous finite-dimensional (unitary) representation of $G$,
and let $\om$ be a central weight on $\G$. For each $n\in \N$, we define
$$n(\om, \pi)=\max \{\om(\tau) \mid \tau \in \supp\, \pi \}.$$
\end{defn}

\begin{prop}\label{P:weight-point der}
Let $\om$ be a central weight on $\G$. Suppose that,
for every $\pi \in \G$,
$$\inf \{ \frac{n(\om,\pi^{\otimes n})}{n} \mid n\in \N \}=0,$$
where $\pi^{\otimes n}=\pi \otimes \cdots \otimes \pi$, $n$-times. Then $A(G,\om)$
has no non-zero continuous point derivation at $e$.
\end{prop}

\begin{proof}
Let $P^1(G)$ denote the set of continuous positive-definite
functions on $G$ such that $\|f\|_{A(G)}=f(e)=1$. Each function in $P^1(G)$ is uniquely
determined by a representation $\pi\in \G$ and an orthonormal vector $\xi\in H_\pi$ so
that $f(x)=\la \pi(x)\xi \mid \xi \ra.$
Hence, in particular, $A(G,\om)$ is the $\|\cdot \|_{A(G,\om)}$-closure of $\spn\{P^1(G)\}$ which is $\fT(G)$.

Now let $d : A(G,\om) \to \C$ be a continuous point derivation at $e$. We will
show that $d=0$ by showing that $d$ vanishes on $P^1(G)$. Let $f\in P^1(G)$, and
$\pi\in \G$ and an orthonormal vector $\xi\in H_\pi$ so
that $f(x)=\la \pi(x)\xi \mid \xi \ra.$ It is routine to verify that
\begin{equation}\label{Eq:point der}
d(f^n)=n[f(e)]^{n-1}d(f)=nd(f) \ \ (n\in \N).
\end{equation}
On the other hand, for each $n\in \N$, by Schur orthogonality relation,
$\widehat{f^n}(\tau)=0$ if $\tau \not\in \supp\, \pi^{\otimes n}$. Thus
\begin{eqnarray*}
\|\widehat{f^n}\|_{A(G,\om)} &=& \sum_{\tau \in \supp\, \pi^{\otimes n}} \om(\tau)d_\tau \|\widehat{f^n}(\tau)\|_1
\\ &\leq &  n(\om, \pi^{\otimes n}) \sum_{\tau \in \supp\, \pi^{\otimes n}}d_\tau \|\widehat{f^n}(\tau)\|_1
\\ &=& n(\om, \pi^{\otimes n}) \|f^n\|_{A(G)}
\\ &=& n(\om, \pi^{\otimes n}),
\end{eqnarray*}
where the last equality follows since $f^n$ is a positive-definite function, and so,
$\|f^n\|_{A(G)}=f^n(e)=1$. Therefore, by (\ref{Eq:point der}),
$$|d(f)|\leq \frac{\|d\|\|f^n\|_{A(G,\om)}}{n} \leq \|d\| \frac{n(\om, \pi^{\otimes n})}{n} \ \ \ (n\in \N).$$
So by hypothesis, $d(f)=0$.
\end{proof}

The preceding proposition was proven in \cite[Proposition 5.1]{Sa} for $G$ abelian. Our extension allow
us to study operator weak amenability for the case when $G$ is non-abelian. The following theorem
is one application of Proposition  \ref{P:weight-point der}. Other applications will be
given in Section \ref{S:weight-SU(2)}.

\begin{thm}
Let $G$ be a compact, totally disconnected group, and let $\om$ be
a central weight on $G$. Then $A(G,\om)$ is operator weakly amenable.
\end{thm}

\begin{proof}
It is well-known that $G$ has a base of the identity consisting of open, normal
compact subgroups of $G$, and so, $G$ is a projective limit of finite groups in the sense of
\cite[Definition 4.1.4]{RS} (see also \cite[Theorem 4.1.14]{RS}). Hence a similar argument
to \cite[Theorem 27.43]{HR2} shows that, for every $\pi \in \G$,
there is a finite group $H$ (which is the quotient of $G$ by some open normal compact subgroup)
so that $\pi \in \widehat{H}$. Hence
$$\{\tau \mid \tau \in \supp\, \pi^{\otimes n}, n\in \N \}\subset \widehat{H}$$
is finite. Thus for the weight $\Om_1(\sg)=\Om(\sg)d_\sg$, $\sg \in \G$ (Remark \ref{rem-central} (2)), we have
$$\inf \{ \frac{n(\Om_1,\pi^{\otimes n})}{n} \mid n\in \N \}=0,$$
for every $\pi \in \G$. Therefore, by Proposition \ref{P:weight-point der},
$A(G,\Om_1)$ has no continuous point derivation at $e$. Since $A(G,\Om_1) \subseteq A_\Del(G,\Om)$,
the same holds for $A_\Del(G,\Om)$. It follows from Theorem \ref{T:Amen-weak Amen-Delta}(iii),
$A(G,\om)$ is operator weakly amenable.
\end{proof}

\subsection{Arens regularity}\label{S:BF alg-compact-Arens regularity}

In this section, we study the Arens regularity of Beurling-Fourier algebras. We
provide classes of Beurling-Fourier algebras on compact groups that either satisfy
or fail the Arens regularity.

\begin{defn}\label{D:Weight-Arens reg-operator}
Let $W$ be a weight on $A(G)$. We denote the bounded operator $\Gamma(W)(W^{-1} \otimes W^{-1})$ by $\TOm$, and we write
	$$\TOm = \bigoplus_{\pi, \sigma \in \widehat{G}}\TOm(\pi, \sigma),$$
where $\TOm(\pi, \sigma) \in M_{d_\pi}\otimes M_{d_\sigma}$.
\end{defn}

The following theorem is proven in \cite[Theorem 8.11]{DL} in the case where $G$ is abelian. We extend
it to the general case and apply it to construct Arens regular Beurling-Fourier algebras
on non-abelian compact groups.

\begin{thm}\label{T:weight-Arens reg-positive}
Let $\om$ be a central weight on $\G$. Suppose that
$$\lim_{\pi\to \infty} \limsup_{\sg\to \infty} \norm{\TOm(\pi, \sigma)}_{M_{d_\pi}\otimes M_{d_\sigma}}
= \lim_{\sg\to \infty} \limsup_{\pi\to \infty} \norm{\TOm(\pi, \sigma)}_{M_{d_\pi}\otimes M_{d_\sigma}}=0.$$
Then $A(G,\om)$ and all its even duals are Arens regular.
\end{thm}

\begin{proof}
Since $A(G,\om)^{**}=A(G,\om)\oplus C^*_r(G, \om^{-1})^\bot$, it suffices to show that
\begin{equation}\label{Eq:Split-dual}
\Phi \bo \Psi =\Phi \diamond  \Psi =0 \ \ (\Phi,\Psi \in C^*_r(G, \om^{-1})^\bot).
\end{equation}
To see this, first note that, for all $n\in \N$,
	$$A(G,\om)^{(2n)}=A(G,\om)\oplus \bigoplus_{i=0}^{n-1} [C^*_r(G, \om^{-1})^\bot]^{(2i)},$$
where $X^{(m)}$, $m\ge 1$ implies $m$-th dual of a Banach space $X$.
Thus if (\ref{Eq:Split-dual}) holds, then both Arens products vanishes on $\bigoplus_{i=0}^{n-1} [C^*_r(G, \om^{-1})^\bot]^{(2i)}$,
which implies that $A(G,\om)^{(2n-2)}$ is Arens regular.
We will now prove (\ref{Eq:Split-dual}). Suppose that
$W$ is the central weight on $\G$ defined in Definition \ref{D:weight-non abelian-central}.
Let $\Phi, \Psi \in C^*_r(G, \om^{-1})^\bot$ with norm 1.
Using the identification (\ref{Eq:weight-A(G,W)}), take two nets $\{f_\alpha \}$ and $\{g_\beta\}$ in $A(G)$
with $\|f_\alpha \|_{A(G)}, \|g_\beta\|_{A(G)}\leq 1$
such that $W^{-1}f_\alpha \to \Phi$ and $W^{-1}g_\beta \to \Psi$ in $w^*$-topology of $VN(G,W^{-1})$.
Let $A\in VN(G)$ with $\|A\|_{VN(G)}\leq 1$. Then, by (\ref{modified-co-multiplication}),
\begin{align*}
\lefteqn{\la \Phi \bo \Psi,\; AW \ra}
\\ &= \lim_{\alpha} \lim_{\beta} \la (W^{-1}f_\alpha)\cdot_{A(G,W)} (W^{-1}g_\beta),\; AW \ra
\\ &= \lim_{\alpha} \lim_{\beta}
\la (W^{-1} \otimes W^{-1})(f_\alpha \otimes g_\beta),\;  \Gamma^W(AW) \ra
\\ &= \lim_{\alpha} \lim_{\beta}
\la f_\alpha \otimes g_\beta, \; \widetilde{\Gamma}\circ \Phi^{-1}(AW) \ra
\\ &= \lim_{\alpha} \lim_{\beta}
\la f_\alpha \otimes g_\beta, \; \Gamma(A)\Gamma(W)(W^{-1} \otimes W^{-1}) \ra
\\ &= \lim_{\alpha} \lim_{\beta}
\sum_{\pi \in \G} \sum_{\sg \in \G} d_\pi d_\sg
 \text{tr}[(\widehat{f}_\alpha(\pi)\otimes \widehat{g}_\beta(\sg))
\Gamma(A)(\pi,\sigma)\TOm(\pi, \sigma)].
\end{align*}
Now let $\epsilon>0$. By hypothesis, there is a finite set $E$ in $\G$
such that for every $\pi \in E^c:=\G \setminus E$, there is a finite set $F$ (depending on $\epsilon$ and $\pi$) in $\G$ for which we have:
\begin{equation}\label{Eq:epsilon}
\|\TOm(\pi, \sigma) \|_{M_{d_\pi}\otimes M_{d_\sigma}} \leq \epsilon \ \ \ (\sg \in F^c).
\end{equation}
Now for every $\A,\B\subset \G$, let
$$ \Xi(\A,\B)=\sum_{\pi \in \A} \sum_{\sg \in \B} d_\pi d_\sg
|\text{tr}[(\widehat{f}_\alpha(\pi)\otimes \widehat{g}_\beta(\sg))\Gamma(A)(\pi,\sigma)\TOm(\pi, \sigma)]|.$$
Then, for every $\alpha$ and $\beta$,
\begin{eqnarray*}
|\la f_\alpha \otimes g_\beta, \; \Gamma(A) \Gamma(W)(W^{-1} \otimes W^{-1})  \ra|
&\leq &  \Xi(\G,\G)
\\ &=&  \Xi(E,\G)
+  \Xi(E^c,F)
+ \Xi(E^c,F^c).
\end{eqnarray*}
We will show that
$$ \lim_{\alpha} \limsup_{\beta}\ \Xi(E,\G)= \lim_{\alpha} \lim_{\beta}\ \Xi(E^c,F)=0 \ \text{and}\ \Xi(E^c,F^c)\leq \epsilon.$$
We have
\begin{eqnarray*}
\Xi(E,\G) &\leq& \sum_{\pi \in E} d_\pi \|\widehat{f}_\alpha(\pi) \|_1 \|g_\beta\|_{A(G)}
\|\Gamma(A)\| \| \TOm \|
\\ &\leq & \sum_{\pi \in E} d_\pi \|\widehat{f}_\alpha(\pi) \|_1
\\ &\leq & \sum_{\pi \in E} d_\pi^{3/2}\sum_{i,j=1}^{d_\pi} |\widehat{f}_\alpha(\pi)_{ij}|^2,
\end{eqnarray*}
where $\widehat{f}_\alpha(\pi)=[\widehat{f}_\alpha(\pi)_{ij}]$.
Since $E$ does not depend on $\alpha$ and $\Phi \in C^*_r(G, \om^{-1})^\bot$,
for all $\pi\in E$ we have
\begin{equation}\label{Eq:Arens-first case}
\lim_{\alpha} \widehat{f}_\alpha(\pi)_{ij} = \lim_{\alpha} \la f_\alpha, \overline{\pi_{ij}} \ra
=\la \Phi,\; \overline{\pi_{ij}} \ra=0,
\end{equation}
where $\pi_{ij}$'s are the trigonometric polynomials defined in (\ref{eq:piij}).
Therefore
	$$\lim_{\alpha} \limsup_{\beta}\ \Xi(E,\G)=0$$
since $E$ is finite. For the second case, note that
\begin{eqnarray*}
\Xi(E^c,F) &\leq& \sum_{\sg \in F}  d_\sg \|f_\alpha\| \|\widehat{g}_\beta(\sg)\|_1
\|\Gamma(A)\| \| \TOm \|
\\ &\leq & \sum_{\sg \in F} d_\sg \|\widehat{g}_\beta(\sg)\|_1.
\end{eqnarray*}
Since $F$ does not depend on $n$ and $\Psi \in C^*_r(G, \om^{-1})^\bot$,
similar to (\ref{Eq:Arens-first case}),
we have
$$\lim_{\beta} \widehat{g}_\beta(\sg)=0$$
for all $\sg\in F$.
Hence, because of finiteness of $F$,
$$ \lim_{\alpha} \lim_{\beta}\ \Xi(E^c,F)=0.$$
Finally
\begin{eqnarray*}
\Xi(E^c,F^c) &\leq& \sum_{\pi \in E^c} \sum_{\sg \in F^c} d_\pi d_\sg \|\widehat{f}_\alpha(\pi) \|_1
\|\widehat{g}_\beta(\sg)\|_1 \|\Gamma(A)\| \| \TOm(\pi, \sigma) \|
\\ &\leq &\epsilon \sum_{\pi \in E^c}  d_\pi \|\widehat{f}_\alpha(\pi) \|_1
\sum_{\sg \in F^c} d_\sg \|\widehat{g}_\beta(\sg)\|_1,
\end{eqnarray*}
where the last inequality follows from (\ref{Eq:epsilon}).
Thus, again since $\|f_\alpha\|_{A(G)}$ and $\|g_\beta\|_{A(G)} \leq 1$ we have
$$ \Xi(E^c,F^c)\leq \epsilon.$$
since $\epsilon$ was arbitrary, it follows that
$$\lim_{\alpha} \lim_{\beta}\  \la f_\alpha \otimes g_\beta, \;  \Gamma(A) \Gamma(W)(W^{-1} \otimes W^{-1} \ra=0.$$
Hence this shows that the first Arens product vanishes on  $C^*_r(G, \om^{-1})^\bot$ and we are done! The proof for the second Arens product is similar.
\end{proof}

\begin{cor}\label{C:weight-Arens regular-Simple Lie group}
Let $G$ be a compact, connected, simple Lie group, and let $\sg_a$ be the central weight on $\G$ defined in \eqref{Eq:weight-non-abelian-ln}.
Then $A(G,\sg_a)$ is Arens regular if $a> 0$.
\end{cor}

\begin{proof}
Let $\Theta_a$ be the corresponding operator defined in Definition \ref{D:Weight-Arens reg-operator}
associated to the weight $\sg_a$. For every $\pi,\rho \in \G$, we have
$$\Theta_a(\pi, \rho)=
\bigoplus_{k=1}^m \frac{\sg_a(\tau_k)}{\sg_a{(\pi)}\sg_a{(\rho})}1_{B(H_{\tau_k})},$$
where $\pi \otimes \rho \cong \bigoplus_{k=1}^m \tau_k$ is the irreducible decomposition
of $\pi \otimes \rho$. Since for each $1\leq k \leq m$, $d_{\tau_k} \leq d_\pi d_\rho$,
it follows that
\begin{eqnarray*}
\sg_a(\tau_k) &=& (1+ \ln d_{\tau_k})^a
\\ &\leq & (1+ \ln d_{\pi})^a + (1+ \ln d_{\rho})^a.
\\ &=& \sg_a{(\pi)} + \sg_a{(\rho}).
\end{eqnarray*}
Thus
\begin{equation}\label{Eq:Arens-simple Lie}
\| \Theta_a(\pi, \rho) \| \leq \frac{1}{(1+ \ln d_{\pi})^a}+ \frac{1}{(1+ \ln d_{\sigma})^a}.
\end{equation}
On the other hand, as it was pointed out in the proof of
Corollary \ref{C:weight-Op amen-simple Lie group}, $d_\pi \to \infty$ as $\pi\to \infty$.
Therefore, from (\ref{Eq:Arens-simple Lie}), it follows that
$$ \lim_{\pi \to \infty} \limsup_{\rho \to \infty} \norm{\Theta_a(\pi, \rho)}
= \lim_{\rho \to \infty} \limsup_{\pi \to \infty} \norm{\Theta_a(\pi, \rho)}=0.$$
Therefore $A(G,\sg_a)$ is Arens regular by Theorem \ref{T:weight-Arens reg-positive}.
\end{proof}

The preceding example dealt with Beurling-Foureir algebras on certain Lie groups.
We can also construct Arens regular Beurling-Fourier algebra on non-abelian totally disconnected
groups. We recall that, for each $n\in \N$, the special linear group $SL(2,2^n)$ denotes the set of
all 2$\times$2 matrix with the determinate 1 on a finite field of $2^n$ elements.
It is well-known that $SL(2,2^n)$ is a finite simple group (see \cite[Section 2.7]{C}).

\begin{cor}
Let $G=\Pi_{n=1}^\infty SL(2,2^n)$,
and let $\sg_a$ be the central weight on $\G$ defined in \eqref{Eq:weight-non-abelian-ln}.
Then $A(G,\sg_a)$ is Arens regular if $a > 0$.
\end{cor}

\begin{proof}
Let $\Theta_a$ be the corresponding operator defined in Definition \ref{D:Weight-Arens reg-operator} associated to the weight $\sg_a$.
Similarly to the proof of Corollary \ref{C:weight-Arens regular-Simple Lie group}, we have
$$\| \Theta_a(\pi, \rho) \| \leq \frac{1}{(1+ \ln d_{\pi})^a}+
\frac{1}{(1+ \ln d_{\sigma})^a} \ \ \ (\pi,\rho \in \G).$$
However it is shown in \cite[Section 2.7]{C} that every non-trival continuous irreducible
unitary representation of $SL(2,2^n)$ has dimension at least $2^n-1$. Thus $d_\pi \to \infty$
as $\pi \to \infty$ on $\G$. Hence
$$ \lim_{\pi \to \infty} \limsup_{\rho \to \infty} \norm{\Theta_a(\pi, \rho)}
= \lim_{\rho \to \infty} \limsup_{\pi \to \infty} \norm{\Theta_a(\pi, \rho)}=0.$$
Therefore $A(G,\sg_a)$ is Arens regular by Theorem \ref{T:weight-Arens reg-positive}.
\end{proof}

We finish this section with the following theorem that presents examples of non-Arens regular Beurling-Fourier algebras on non-abelian compact groups.

\begin{thm}
Let $\{G_i \}_{i\in I}$ be an infinite family of non-trivial compact groups,
and let, for each $i\in I$, $\om_i$ be a central weight on $\G_i$.
Let $G=\prod_{i\in I} G_i$ and $\om=\prod_{i\in I} \om_i$.
Suppose further that $\om$ is bounded away from zero. Then $A(G,\om)$ is not Arens regular.
\end{thm}

\begin{proof}
If $J\subseteq I$, then it is clear that $A(\prod_{i\in J} G_i, \prod_{i\in J} \om_i)$ is a quotient of $A(G,\om)$.
Thus it suffices to prove the statement of the theorem when $I$ is infinite and countable. So we assume
that $I=\N$. For each $i\in \N$, let $\pi_i\in \G_i$ be a non-trivial representation. For each $m,n\in \N$,
let $u_m$ and $v_n$ be elements of $\G$ defined by
$$u_m(\widetilde{x})=\pi_{2m}(x_{2m}) \ \ , \ \ v_n(\widetilde{x})=\pi_{2n+1}(x_{2n+1}) \ \ \ \ (\widetilde{x}=\{x_i\}_{i\in \N}\in G).$$
We have
\begin{equation}\label{Eq:non Arens-product}
u_m\otimes v_n=
              \begin{cases}
              \pi_{2m}\times \pi_{2n+1} & \text{if}\; 2m < 2n+1\\
             \pi_{2n+1}\times \pi_{2m}  & \text{if}\; 2m > 2n+1.
              \end{cases}
\end{equation}
In particular, $\{ u_m\otimes v_n \}_{m,n\in \N}$ are distinct elements of $\G$. Now let
$$f_m=\frac{1}{d_{u_m}} \chi_{u_m}\ \ \text{and}\ \ g_n=\frac{1}{d_{u_n}} \chi_{u_n},$$
where $\chi_\pi$ is the character of $\pi\in \G$ i.e. $\chi_\pi(x)=\text{tr}[\pi(x)]$.
Let $A\in VN(G)$ such that for every $\pi\in \G$,
\begin{equation*}\label{Eq:non Arens-product}
\F(A)(\pi)=
              \begin{cases}
              1_{M_{d_{u_m}}\otimes M_{d_{v_n}}} & \ \ \text{if}\ \pi=u_m \otimes v_n \ \text{and}\ 2m < 2n+1,\\
              0  & \ \ \text{otherwise}.
              \end{cases}
\end{equation*}
Suppose that
$W$ is the central weight on $\G$ defined in Definition \ref{D:weight-non abelian-central}.
Then, by Definition \ref{D:Weight-Arens reg-operator}, (\ref{Eq:co-multi}),
(\ref{Eq:bdd co prod-compact}), and (\ref{Eq:non Arens-product}), we have
\begin{eqnarray*}
\Theta(u_m, v_n) &=& [\Gamma(W)(W^{-1}\otimes W^{-1})] (u_m, v_n) \\
&=&  \bigoplus_{ \sigma \in\, \supp\, u_m\otimes v_n} \om(\sigma) \om(u_m)^{-1} \om(v_n)^{-1}\om(u_m\otimes v_n) 1_{M_{d_\sigma}} \\
&=& \om(u_m)^{-1} \om(v_n)^{-1}\om(u_m\otimes v_n)1_{M_{d_{u_m}}\otimes M_{d_{v_n}}}\\
&=& 1_{M_{d_{u_m}}\otimes M_{d_{v_n}}}.
\end{eqnarray*}
Thus,  by (\ref{modified-co-multiplication}), for every $m,n\in \N$,
\begin{eqnarray*}
& & \la (W^{-1}f_m)\cdot_{A(G,W)} (W^{-1}g_n),\; AW \ra
\\ &=& \la (W^{-1} \otimes W^{-1})(f_m \otimes g_n),\;  \Gamma^W(AW) \ra
\\ &=&
\la f_m \otimes g_m, \; \widetilde{\Gamma}\circ \Phi^{-1}(AW) \ra
\\ &=&
\la f_m \otimes g_n, \; \Gamma(A)\Gamma(W)(W^{-1} \otimes W^{-1}) \ra
\\ &=& d_{u_m} d_{v_n} \text{tr}[(\widehat{f}_m(u_m)\otimes \widehat{g}_n(v_n)
\Gamma{(A)}(u_m , v_n)\TOm(u_m , v_n)]
\\ &=&  \begin{cases}
              1                & \text{if}\; 2m < 2n+1\\
              0                & \text{if}\; 2m > 2n+1.
              \end{cases}
\end{eqnarray*}
Therefore the repeated limit of $(W^{-1}f_m)\cdot_{A(G,W)} (W^{-1}g_n),\; AW \ra$ exits but they are not equal. This shows that $A(G,\om)$ is not Arens regular.
\end{proof}

\begin{rem}
We finish this section by pointing out that the proof of the preceding theorem can be adapted,
with almost the same approach, to show that if $\{G_i \}_{i\in I}$ is an infinite family of non-trivial compact groups and
if $\om_a$ is the weight (\ref{Eq:weight-non-abelian-dim}) on the dual of $\prod_{i\in I}G_i$,
then $A(G,\om_a)$ is not Arens regular for any $a \geq 0$.
\end{rem}

\section{The $2\times 2$ special unitary group}\label{S:weight-SU(2)}

In this section, we apply the results of the preceding section to study explicitly
the behavior of Beurling-Fourier algebras on $2\times 2$ special unitary group:
$$SU(2)=\{ A\in M_2(\C) \mid A \ \text{is unitary},\ \det A=1 \}.$$
First we make the following important observation which allows us to correspond various
central weights on $\widehat{SU(2)}$ to their restriction on $\Z$.

\subsection{Restriction of the weight on $\Z$}
Let $\om$ be a central weight on $\widehat{SU(2)}$. We can assume that $\T$ is a closed subgroup of $SU(2)$ by the identification
\[e^{it} \mapsto
\begin{array}{ll}
& \left[
\begin{array}{ll}
e^{it} & 0 \\
0 & e^{-it}
\end{array}
\right]
\end{array} \ \  (t\in [0,2\pi]).
\]
By the representation
theory of $SU(2)$ \cite[29.18 and 29.20]{HR2},
$$\widehat{SU(2)}=\{ \pi_l \mid l=0,1/2, 1, 3/2, \ldots \},$$
and $\text{dim}\, \pi_l=2l+1$. Moreover,
$$\pi_l (e^{it})=diag(e^{i2lt}, e^{i2(l-1)t},\ldots , e^{-i2lt}) \ \ \ \ (t\in [0,2\pi]).$$
Let $n\in \Z$, and let $\chi_n$ be the character on $\T$ defined by
$$\chi_n(e^{it})=e^{int}  \ \ \ \ (t\in [0,2\pi]).$$
Then $\pi_l$ is an extension of $\chi_n$ if and only if $|n|\leq 2l$. Hence if we identify $\Z$ with $\widehat{\T}$
through the Plancheral map $n \mapsto \chi_n$, we have
\begin{equation}\label{Eq:SU(2)-T}
	\om_\T(n)=\inf \{\om(\pi_l) \mid |n|\leq 2l \} \ \ \ (n\in \Z).
\end{equation}
This, in particular, implies that $A(\T,\om_\T)$ is a complete quotient
of $A(SU(2),\om)$ from Proposition \ref{P:weight-subgroup}. We will show in the following sections that $A(SU(2),\om)$ behave very similarly
to that of $A(\T, \om_\T)$.

\begin{exm}\label{E:weight-SU(2)}
Let $a \geq 0$ and $0 \leq b \leq 1$. We define the functions
$\sg_a$, $\om_a$, $\rho_b$ from $\widehat{SU(2)}$ into $[1,\infty)$ by
$$ \sg_a(\pi_l)=(1+\ln (2l+1))^a \ \ \ (\pi_l \in \widehat{SU(2)}),$$
$$ \om_a(\pi_l)= (2l+1)^a \ \ \ (\pi_l \in \widehat{SU(2)}),$$
$$ \rho_b(\pi_l)=e^{ (2l+1)^b} \ \ \ (\pi_l \in \widehat{SU(2)}).$$
Since $d_{\pi_l}=2l+1$, the first two weights are the one defined in Example \ref{E:weight-compact}.
Also we know from \cite[29.29]{HR2} that, for every $l,r=0,1/2, 1, 3/2, \ldots$,
\begin{equation}\label{Eq:Tensor formula}
\pi_l \otimes \pi_r \cong \pi_{|l-r|}\oplus \pi_{|l-r|+1} \oplus \cdots \oplus \pi_{|l+r|}=\bigoplus_{k=|l-r|}^{|l+r|} \pi_k.
\end{equation}
Therefore it is routine to verify that $\rho_b$ also defines a weight on $\widehat{SU(2)}$. Moreover, by
(\ref{Eq:SU(2)-T}), the restriction of the above weights on $\widehat{\T}=\Z$ corresponds to the following well-known weights on $\Z$:
$$ \sg'_a(n)=(1+\ln (1+|n|))^a \ \ \ (n\in \Z),$$
$$ \om'_a(n)= (1+|n|)^a \ \ \ (n\in \Z),$$
$$ \rho'_b(n)=e^{ (1+|n|)^b} \ \ \ (n\in \Z).$$
\end{exm}

\subsection{Operator amenability and weak amenability}

Let $ a \geq 0$ and $0 \leq b \leq 1$, and let $ \sg'_a$, $ \om'_a$, and
$\rho'_b$ be the weights on $\Z$ defined in Example \ref{E:weight-SU(2)}.
N. Gr{\o}nb{\ae}k has characterized in \cite{G1} and \cite{G2}
when either of $A(\T, \sg'_a)$, $A(\T, \om'_a)$, or  $A(\T,\rho'_b)$ is amenable
or weakly amenable. We summarized them below:

(i) $A(\T, \sg'_a)$ or $A(\T, \om'_a)$ is amenable if and only if $a=0$;

(ii) $A(\T,\rho'_b)$ is amenable if and only if $b=0$;

(iii) $A(\T, \sg'_a)$ is always weakly amenable;

(iv) $A(\T, \om'_a)$ has no non-zero continuous point
derivation at $0$ if and only if $0\leq a <1$;

(v)  $A(\T, \om'_a)$ is weakly amenable if and only if $0\leq a <1/2$;

(vi) $A(\T,\rho'_b)$ has non-zero continuous point
derivations at $0$ if $b>0$;

(vii) $A(\T,\rho'_b)$ is never weakly amenable unless $b=0$.\\
We will show in the following theorem that, in most cases, the analogous of these results holds
for the corresponding weights on $\widehat{SU(2)}$.

\begin{thm}
Let $a \geq 0$ and $0 \leq b \leq 1$, and let $ \sg_a$, $ \om_a$, and
$\rho_b$ be the weights on $\widehat{SU(2)}$ defined in Example \ref{E:weight-SU(2)}.
Then the following holds:\\
$(i)$ $A(SU(2), \sg_a)$ or $A(SU(2), \om_a)$ is operator amenable if and only if $a=0$;\\
$(ii)$ $A(SU(2),\rho_b)$ is  operator amenable if and only if $b=0$;\\
$(iii)$ $A(SU(2), \sg_a)$ has no non-zero continuous point derivation at $e$;\\
$(iv)$ $A(SU(2), \om_a)$ has no non-zero continuous point derivation at $e$ if $0\leq a <1$;\\
$(v)$ $A(SU(2), \om_a)$ is not operator weakly amenable if $a \geq 1/2$;\\
$(vi)$ $A(SU(2),\rho_b)$ is never  operator weakly amenable unless $b=0$.
\end{thm}

\begin{proof}
(i) and (ii). If $a=b=0$, then these Beurling-Foureir algebras are
$A(SU(2))$. Thus the result follows from \cite{Ruan}.
On the other hand, if $a, b >0$, then
$$\lim_{\pi_l \to \infty}  \sg_a(\pi_l)=\lim_{\pi_l \to \infty}  \om_a(\pi_l)=\lim_{\pi_l \to \infty}  \rho_\alpha(\pi_l)
=\infty.$$
Therefore by Theorem \ref{T:weight-op. amen}, neither of $A(SU(2), \sg_a)$, $A(SU(2), \om_a)$,
nor $A(SU(2),\rho_b)$ is operator amenable.\\
(iii) and (iv). It follows from the tensor formula (\ref{Eq:Tensor formula}) and Schur orthogonality relation that the conjugate of any representation $\pi_l$ is itself.  Moreover, for every $n\in \N$, and $\pi_l \in \widehat{SU(2)}$, we have
$$n(\sg_a,\pi_l^{\otimes n})\leq (1+\ln n + \ln (2l+1))^a \ \ , \ \ n(\om_a,\pi_l^{\otimes n})\leq [n(2l+1)]^a.$$
Therefore
$$\inf \{ \frac{n(\sigma_a, \pi_l^{\otimes n})}{n} \mid n\in \N \}=0$$
for all $a \geq 0$ and
$$\inf \{ \frac{n(\omega_a, \pi_l^{\otimes n})}{n} \mid n\in \N \}=0$$
when $0\leq a < 1$. Thus the results follow from Proposition \ref{P:weight-point der}.\\
(v) and (vi). As it was pointed out in Example \ref{E:weight-SU(2)},
$$A(SU(2), \sg_a){|_\T}=A(\T, \sg'_a) \ \ \text{and} \ \ A(SU(2), \om_a){|_\T}=A(\T, \om'_a). $$
Hence operator weak amenability of $A(SU(2), \sg_a)$ and $A(SU(2), \om_a)$ implies
the weak amenability of $A(\T, \sg'_a)$ and $A(\T, \om'_a)$, respectively.
Thus it follows from \cite{G1} that $A(SU(2), \om_a)$ is operator weakly amenable only if $0\leq a <1/2$
and $A(SU(2),\rho_b)$ is never operator weakly amenable unless $b=0$.
\end{proof}

\subsection{Connection with the amenability of $A(SU(2))$}\label{S:connection amen A(SU(2))}

B. E. Johnson in his memoirs \cite{J2} in 1972 introduced the concept of an amenable
Banach algebra and proved his famous theorem: the group algebra $L^1(G)$ is amenable
if and only if $G$ is amenable. It was believed that similar conclusion holds for
the Fourier algebra $A(G)$ since $A(G)$ acts in lots of cases like a dual of $L^1(G)$.
However it was Johnson himself who proved a remarkable result that $A(SU(2))$ is not amenable
\cite{J1}. Shortly after Ruan showed in \cite{Ruan} that the amenability of $G$ corresponds exactly
to the ``operator amenability" of $A(G)$ which led to the several applications of operator spaces
to harmonic analysis. Later on, Forrest and Runde \cite{FR} settled the question of amenability for the Foureir algebras: $A(G)$ is amenable if and only if $G$ has an abelian subgroup of finite index.

We would like to analyze the non-amenability of $A(SU(2))$ and show its connection
with the Beurling-Fourier algebras on $\widehat{SU(2)}$ and the classical Beurling algebras
on $\Z$. Johnson used the properties of the algebra $A_\gamma(G)$ defined in (\ref{Eq:A-DEl-formula}) in a very clear way to obtain his result. His approach was widely generalized and studied in \cite{FSS1}.
By \cite[Theorem 3.2]{J1} (see also \cite[Corollary 1.5]{FSS1}) the amenability of $A(G)$ implies
that the maximal ideal $\{f\in A_\gamma(SU(2)) \mid f(e)=0 \}$ of $A_\gamma(G)$ has a bounded
approximate identity.
However $A_\gamma(G)$ is nothing but the Beurling-Fourier algebra $A(G,\om_1)$ in (\ref{Eq:weight-non-abelian-dim}) where
$\om_1(\pi_l)=1+2l$ for all $\pi_l\in \widehat{SU(2)}$ (see (\ref{Eq:norm-gamma})).
Since, by Example \ref{E:weight-SU(2)}, the restriction of $A(SU(2),\om_1)$ on $\T$ is $A(\T, \om'_1)$,
this implies that $\{f\in A(\T, \om'_1) \mid f(e)=0 \}$ has a bounded approximate identity.
Hence by Theorem \ref{T:Amen-weak Amen-Delta}, $A(\T, \om'_1)$ is amenable which is impossible
by \cite{G2}. This argument can also be applied to the question of weak amenability because again
by a similar argument, the weak amenability of $A(SU(2))$ implies that $A(\T, \om'_{1/2})$ is weakly amenable
which is shown to fail in \cite{G1}.

As we see, the preceding arguments shows that the (weak) amenability of the Fourier algebra $A(SU(2))$
is closely related to the (weak) amenability of a well-known Beurling algebra on $\Z$ and it has inherit
connection. That is why $A(SU(2))$ fails to be amenable or even weakly amenable because the Beurling
algebras are known not to behave well with regard to cohomology.
We believe that these connections are non-trivial and certainly worthwhile investigating more.
For example, it is shown in \cite{FSS2} that $A(G)$
is not weakly amenable if $G$ is compact, connected, and non-abelain. If we assume further that $G$
is a Lie group, then again amenability or weak amenability of $A(G)$ relates closely to the behavior
of certain Beurling algebras on a maximal torus of $G$. By investigating more this relation,
we might be able to have a better understanding of the structure of Fourier algebras on compact Lie groups.

\subsection{Arens regularity}\label{S:Arens regularity-SU(2)}

Let $a \geq 0$ and $0 \leq b \leq 1$, and let $ \sg'_a$, $ \om'_a$, and
$\rho'_b$ be the weights on $\Z$ defined in Example \ref{E:weight-SU(2)}. As it is shown
in \cite[Theorem 8.11]{DL},  $A(\T, \sg'_a)$, $A(\T, \om'_a)$, or  $A(\T,\rho'_b)$ are Arens regular for $a, b >0$.
We will show in the following theorem that the exact analogous of these results holds
for the corresponding weights on $\widehat{SU(2)}$.

\begin{thm}\label{T:weight-SU(2)-Arens regular}
Let $a \geq 0$ and $0 \leq b \leq 1$, and let $ \sg_a$, $ \om_a$, and
$\rho_b$ be the weights on $\widehat{SU(2)}$ defined in Example \ref{E:weight-SU(2)}.
Then:\\
$(i)$ $A(SU(2), \sg_a)$ or $A(SU(2), \om_a)$ is Arens regular if and only if $a >0$,\\
$(ii)$ $A(SU(2),\rho_b)$ is Arens regular if and only if $b>0$.
\end{thm}

\begin{proof}
If $a=b=0$, then these Beurling-Foureir algebras are $A(SU(2))$. Thus the result follows from \cite{B1}.
For the converse, suppose that $a>0$. Since $SU(2)$ is a compact, connected, simple Lie group,
it follows from Corollary \ref{C:weight-Arens regular-Simple Lie group} that $A(SU(2), \sg_a)$ is Arens regular.
For the case of $\om_a$, let
	$$W_a = \bigoplus_{\pi\in \G} \om_a(\pi) 1_{B(H_\pi)},$$
and let $\Theta_a$ be the Fourier transform of $\Gamma(W_a) (W_a^{-1} \otimes W_a^{-1})$
(see Definition \ref{D:Weight-Arens reg-operator}).
For every $\pi_l,\pi_r \in \widehat{SU(2)}$, we have
$$\Theta_a(\pi_l, \pi_r)=
\bigoplus_{k=|l-r|}^{|l+r|} \frac{\om_a(\pi_k)}{\om_a{(\pi_l)}\om_a{(\pi_r})}1_{B(H_{\tau_k})},$$
where $\pi_l \otimes \pi_r \cong \bigoplus_{k=|l-r|}^{|l+r|} \pi_k$ is the irreducible decomposition
of $\pi_l \otimes \pi_r$ (see the tensor formula (\ref{Eq:Tensor formula})).
Since $d_{\pi_k} \leq d_{\pi_l} + d_{\pi_r}$, it follows that
	$$\om_a(\pi_k) \leq \om_a{(\pi_l)} + \om_a{(\pi_r}).$$
Thus
	$$\| \Theta_a(\pi_l, \pi_r) \| \leq \frac{1}{(1+ 2l)^a}+ \frac{1}{(1+2r)^a}.$$
Hence it follows that
$$ \lim_{\pi_l \to \infty} \limsup_{\pi_r \to \infty} \norm{\Theta_a(\pi_l, \pi_r)}
= \lim_{\pi_r \to \infty} \limsup_{\pi_l \to \infty} \norm{\Theta_a(\pi_l, \pi_r)}=0.$$
Therefore $A(G,\om_a)$ is Arens regular by Theorem \ref{T:weight-Arens reg-positive}.
The proof of the Arens regularity of $A(SU(2),\rho_b)$ when $b>0$ is similar to the preceding case.
\end{proof}

\subsection{Arens regular subalgebras of Fourier algebras}\label{S:Arens regular-Fourier alg-SU(2)}

It is shown in \cite{Gra} that there are closed ideals in $L^1(\T^n)\cong A(\Z^n)$ ($n\in \N$) that are Arens regular.
Since these ideals are non-unital, by \cite{U1}, they can not have bounded approximate identity.

In this section, we show that we can construct unital, infinite-dimensional Arens regular
closed subalgebras of Fourier algebras on certain products
of $SU(2)$. This goes parallel to the main result of \cite{U1} since these subalgebras are not ideals.
In fact, they are of the form of Beurling-Fourier algebras on $SU(2)$. This is a surprising
and at the same time an interesting result since classical Beurling algebras can never be a closed subalgebra of
a group algebra unless the weight is trivial. However, the relation (\ref{Eq:norm-gamma}) shows that this can happen for
Beurling-Fourier algebras on certain non-abelian groups. More precisely, it is shown in \cite{M} that for a
locally compact group $G$, $\sup \{ d_\pi \mid \pi\in \G \}$ is finite if and only if $G$ is almost abelian i.e. $G$ has an abelian
subgroup of finite index. Thus if $G$ is not almost abelian, then $A_{\gamma^n}(G)$ defined in (\ref{Eq:A-DEl-formula}) is a Beurling-Fourier algebra with growing weight (see also (\ref{Eq:weight-non-abelian-dim}) and (\ref{Eq:norm-gamma})). We will see in the following theorem that this algebras can be Arens regular.

\begin{thm}
Let $n\in \N$, and let $G_n=SU(2)\times \cdots \times SU(2)$, $2^n$-times. Then
$A_{\gamma^n}(SU(2))$ is a unital, infinite-dimensional Arens regular closed subalgebra of the Fourier algebra $A(G_n)$.
\end{thm}

\begin{proof}
Consider the central weight
$$\om_{2^n}(\pi_l)=d_{\pi_l}^{2^n} \ \ \ (\pi_l \in \widehat{SU(2)}).$$
Then, by (\ref{Eq:weight-non-abelian-dim}), (\ref{Eq:A-DEl-formula}), and (\ref{Eq:norm-gamma}),
 $A_{\gamma^n}(SU(2))=A(SU(2),\om_{2^n})$ is a unital, infinite-dimensional closed subalgebra of
$A(G_n)$ and by Theorem \ref{T:weight-SU(2)-Arens regular}, it is Arens regular.
\end{proof}

\end{document}